\documentclass[article,11pt]{amsart}
\usepackage[top=25mm,bottom=25mm,left=25mm,right=25mm]{geometry}
\usepackage{graphicx}
\usepackage{float}

\usepackage{hyperref}
\usepackage{color}

\newtheorem{lemma}{Lemma}[section]

\newtheorem{proposition}{Proposition}[section]
\newtheorem{theorem}{Theorem}[section]
\newtheorem{corollary}{Corollary}[section]
\newtheorem{definition}{Definition}[section]
\newtheorem{remark}{Remark}[section]

\theoremstyle{definition}
\newtheorem{example}{Example}[section]

\makeatletter
\def\section{\@startsection{section}{1}%
\z@{1\linespacing\@plus\linespacing}{1\linespacing}%
{\bf\centering}}
\def\subsection{\@startsection{subsection}{0}%
\z@{\linespacing\@plus\linespacing}{\linespacing}%
{\bf}}
\makeatother

\makeatletter
\@addtoreset{equation}{section}
\makeatother

\DeclareMathOperator{\diam}{diam}

\DeclareMathOperator{\sgn}{sgn}
\DeclareMathOperator{\modulo}{mod}

\DeclareMathOperator{\Id}{Id}

\newcommand{\cA}{\mathcal{A}}

\newcommand{\cB}{\mathcal{B}}

\newcommand{\R}{\mathbb{R}}
\newcommand{\Z}{\mathbb{Z_+}}
\newcommand{\N}{\mathbb{N}}

\newcommand{\Zwithneg}{\mathbb{Z}}

\begin{document}
\title[Reflected Brownian motion on simple nested fractals]
{Reflected Brownian motion on simple nested fractals}
\author{Kamil Kaleta, Mariusz Olszewski and Katarzyna Pietruska-Pa{\l}uba}

\address{K. Kaleta and M. Olszewski \\ Faculty of Pure and Applied Mathematics, Wroc{\l}aw University of Science and Technology, Wyb. Wyspia\'nskiego 27, 50-370 Wroc{\l}aw, Poland}
\email{kamil.kaleta@pwr.edu.pl, mariusz.olszewski@pwr.edu.pl}

\address{K. Pietruska-Pa{\l}uba \\ Institute of Mathematics \\ University of Warsaw
\\ ul. Banacha 2, 02-097 Warszawa, Poland}
\email{kpp@mimuw.edu.pl}

\begin{abstract}
{We prove the existence of the reflected diffusion on a complex of an arbitrary size for a
large class of planar simple nested fractals. Such a process is obtained as a folding projection of the
free Brownian motion from the unbounded fractal. We give sharp necessary geometric conditions on
the fractal under which this projection can be well defined. They are illustrated by various specific
examples. We first construct a proper version of the transition probability densities for reflected
process and we prove that it is a continuous, bounded and symmetric function which satisfies the
Chapman-Kolmogorov equations. These provide us with further regularity properties of the reflected process
such us Markov, Feller and strong Feller property }


\bigskip
\noindent
2010 {\it MS Classification}: Primary: 60J60, 28A80; Secondary: 60J25.
\end{abstract}

\footnotetext{Research was supported by the National Science Center, Poland, grant no. 2015/17/B/ST1/01233 and by the Alexander von Humboldt Foundation, Germany.}

\maketitle

\baselineskip 0.5 cm

\bigskip\bigskip

\section{Introduction}
Stochastic processes on fractals are new a well-established part of probability theory. Rigorous definition of the Brownian motion on the Sierpi\'{n}ski gasket was given by Barlow and Perkins \cite{bib:BP}, and on nested fractals -- by Lindstr{\o}m \cite{bib:Lin}, Kusuoka \cite{bib:Kus2}, Kumagai \cite{bib:Kum}, Fukushima \cite{bib:Fuk1} and others.
  For a fair account of the theory of Brownian motion on simple nested fractals we refer to \cite{bib:Bar} and references therein. The Brownian motion on bounded nested fractals is unique up to a linear change of time (Barlow and Perkins \cite{bib:BP} for the gasket, Sabot \cite{bib:Sab} in the general case). Similar property is true also in the non-nested Sierpi\'{n}ski carpet, see \cite{bib:BBaKT}.
	
For the gasket, the initial definition of \cite{bib:BP} dealt with the process on the infinite set, but the subsequent papers were concerned rather with the process on a finite state-space. In general, it is a standard fact that the diffusion process on an infinite fractal $\mathcal{K}^{\left\langle \infty \right\rangle}:= \bigcup_{M=0}^{\infty} L^M \mathcal{K}^{\left\langle 0\right\rangle}$ can be constructed from the Brownian motion on its bounded counterpart $\mathcal{K}^{\left\langle 0\right\rangle}$ by means of Dirichlet forms \cite{bib:F}. In the present paper, motivated by further applications to fractal models of disordered media, we follow an opposite path: starting with a process on the infinite fractal, we construct a family
of processes on finite fractals $\mathcal K^{\langle M\rangle}=L^M\mathcal K^{\langle 0 \rangle}$. To this goal, we first find sharp geometric conditions on an unbounded planar simple nested fractal $\mathcal{K}^{\left\langle \infty \right\rangle}$ under which the canonical \emph{folding projection} of this set onto $\mathcal{K}^{\left\langle M \right\rangle}:= L^M \mathcal{K}^{\left\langle 0\right\rangle}$  is well defined for every $M \in \Zwithneg$. Then, given the Brownian motion on $\mathcal{K}^{\left\langle \infty \right\rangle}$, we  use this projection to construct an infinite-lifetime (conservative) diffusion process on the bounded fractal $\mathcal{K}^{\left\langle M \right\rangle}$ which we call the \emph{reflected Brownian motion on } $\mathcal{K}^{\left\langle M \right\rangle}$ .

Fractal sets serve as a useful description of the state-space in mathematical physics, percolation theory and crystalography. The existence of a conservative Markov process on a given compact set (in present setting: on a compact fractal) is crucial in many applications.
 Motivations for  this particular project come from a study of some random models with fractal state-spaces,  mainly the fractal counterpart of the so-called Parabolic Anderson Model (PAM), see \cite{bib:Kon-Wolff}) and related objects of spectral theory. In this spirit, the prominent example is the \emph{integrated density of states} (IDS) - one of the most important objects in the large-scale quantum mechanics (see \cite{bib:Car-Lac}).  In the models we are interested, one considers a massless particle which evolves in random environment on an unbounded fractal $\mathcal{K}^{\left\langle \infty \right\rangle}$. The randomness comes from the interaction with an external force field, described by its potential $V_{\omega}$. The motion of the particle itself is modeled by a Markov process which is stochastically independent of $V_{\omega}$. This leads us to the study of the Schr\"odinger-type random Hamiltonians $H_{\omega}= H_0+V_{\omega}$, where $H_0$ is the 'free' Hamiltonian describing the kinetic energy of the particle, and $V_{\omega}$ is the random multiplication operator representing the potential energy of the system (the evolution of such a system is then described by an appropriate one-parameter Feynman-Kac semigroup of operators with respect to the underlying Markov process on $\mathcal{K}^{\left\langle \infty \right\rangle}$). The spectral properties of such infinite-volume (i.e. defined with the whole fractal $\mathcal{K}^{\left\langle \infty \right\rangle}$) Schr\"odinger operators are usually difficult to handle (note that the spectrum of $H_{\omega}$ is typically not discrete). To overcome these obstacles, one needs to approximate the infinite-volume system by the finite-volume ones. More precisely, one first needs to constrain the system to finite-volume state spaces $\mathcal{K}^{\left\langle M \right\rangle}$, and then let $M \to \infty$. Since our input is fully probabilistic, such a plan requires a sequence of Markov processes on bounded fractals $\mathcal{K}^{\left\langle M \right\rangle}$ with infinite lifetime, with clearly
established relations between the processes on consecutive levels (i.e. on bounded fractals $\mathcal{K}^{\left\langle M \right\rangle}$ with increasing sizes). In order to make this plan feasible, these processes should be constructed from the initial process given on $\mathcal{K}^{\left\langle \infty \right\rangle}$.
While in regular, homogeneous, spaces (like $\mathbb R^d$) one can use just the usual projections of the infinite process onto tori (boxes) of increasing sizes (see e.g. \cite{bib:Szn1}), on fractals the situation is more delicate. Even in the case of planar Sierpi\'nski gasket such a naive projection would destroy the Markov property and further regularity properties of the resulting processes. This shows that on fractals a different approach is needed.

An alternative construction for the Sierpi\'{n}ski gasket in $\mathbb R^2$, leading to the reflected Brownian motion, was first proposed in
\cite{bib:KPP-PTRF}. Later, it was extended to the subordinate Brownian motions on the gasket and used in proving the existence and  asymptotic properties of the IDS for such processes in presence of the Poissonian random field \cite{bib:KaPP2, bib:KaPP}. We want to emphasize that this was done exactly along the approximation scheme described above and that the reflected process was indeed a key tool in these investigations (see e.g. the crucial monotonicity argument for the Feynman-Kac functionals, involving the periodized potentials, with respect to the reflected processes on $\mathcal{K}^{\left\langle M \right\rangle}$ in \cite[Th. 3.1]{bib:KaPP2} and \cite[Lem. 4.4-4.5]{bib:KaPP}, the trace estimates in \cite[Prop. 3.1 and Lem. 3.2]{bib:KaPP2}, and the weak scaling of eigenvalues in \cite[Lem 4.3]{bib:KaPP}).
 In the present paper, we generalize the construction from \cite{bib:KPP-PTRF} and prove the existence and further properties of the reflected Brownian motion on $\mathcal{K}^{\left\langle M \right\rangle}$ for a large class of planar simple nested fractals. Sharp estimates of the densities for such a process are given in the companion article \cite{bib:O}. Our present results will allow us to continue the research on the IDS for subordinate Brownian motions evolving in the presence of random potentials on planar nested fractals. This is a primary motivation for our investigations in this paper.

Our approach hinges on a clever labeling of the vertices of the infinite fractal, which we call `good'; fractals permitting for such a labeling are said to have the {\em Good Labeling Property, GLP} in short (Definitions \ref{def:glp} and \ref{def:glp_gen}). Not every planar fractal has GLP, e.g. the Lindstr{\o}m snowflake (Example \ref{ex:snow}) has not -- this is the reason why we exclude this set from our considerations. In Section \ref{subsec:glp} we present the concept of good labelling, and we give an easy-to-check sufficient condition for it to hold (Proposition \ref{pro:glp}). The example of the Lindstr{\o}m snowflake shows again that this condition is sharp.
It is worth mentioning that the GLP is a rather delicate property which essentially depends on the geometry of the fractal (Remark \ref{rem:rem_glp} and Proposition \ref{pro:uniquelabel}). It simplifies in the case of the planar Sierpi\'nski gasket. Note that our definition of the GLP makes sense thanks to the basic result, which says that the vertices of any complex in a simple nested fractal form a regular polygon (Proposition \ref{pro:plane}). Such a geometric property has been conjectured before by some experts in the field, but to the best of our knowledge, the formal proof of this fact was not known. The concept of GLP naturally leads to the 'folding' projection $\pi_M$ of order $M$ from the unbounded fractal $\mathcal{K}^{\left\langle \infty \right\rangle}$ onto $\mathcal{K}^{\left\langle M \right\rangle}$. Its definition and further properties are studied in Section \ref{subsec:proj}. In Section \ref{subsec:suf_cond_glp}, we review  various classes of planar simple nested fractals for which the GLP holds. We prove that all fractals whose building blocks are triangles or squares have the GLP (Theorem \ref{th:triangles}, Corollary \ref{coro:squares}). The same is true 
if all fixed points are essential (
Theorem \ref{th:all_ess}). Moreover, we found a nice full geometric characterization of the GLP for the sets with an even number of essential fixed points (Theorem \ref{th:even_ess}). Note that this also fully explains why the Lindstr{\o}m snowflake is a negative example. All these results taken together show that the class of nested fractals having the GLP is very rich.

For fractals having the GLP, once the labeling is introduced and the projection is well defined, we can pass to the definition of the reflected Brownian motion and its properties (Section \ref{sec:reflected}). The reflected diffusion on $\mathcal{K}^{\left\langle M \right\rangle}$ is defined canonically as a `folding' projection of the `free' Brownian motion from $\mathcal{K}^{\left\langle \infty \right\rangle}$. Its measure is defined by a consistent family of finite dimensional distributions, which guarantees the existence of the corresponding stochastic process. The actual problem we address in the present paper is concerned with the regularity of this process. More precisely, we construct a version of the densities $g_M(t,x,y)$ for its one-dimensional distributions and show that in fact they define the transition probability densities for the process. We prove even more. In Theorem \ref{thm:main1} we obtain that $g_M(t,x,y)$ are symmetric functions in $(x, y)$, which satisfy the Chapman-Kolmogorov equations, have further continuity and boundedness properties, and define a Feller and strong Feller semigroup of operators. In consequence, the resulting reflected process is a symmetric strong Markov process having both Feller and strong Feller properties (Theorem \ref{thm:main2}).
Let us emphasize that all these regularity properties require a rather intricate definition of the densities $g_M(t,x,y)$. We found that the correct one is given by \eqref{eq:refldens}. This formula strongly depends on whether $y \in \mathcal{K}^{\left\langle M \right\rangle}$ is a vertex or not. In the first situation, it involves in an essential way the rank of points $y'$ from the fiber $\pi_M^{-1}(y)$ of $y \in \mathcal{K}^{\left\langle M \right\rangle}$ (by $\textrm{rank}(y')$ of a vertex $y'$ we understand the  number of $M$-complexes meeting at this point). This difficulty is the most critical point for our study. Indeed, due to the geometric properties of nested fractals, for any vertex $y',$  $\textrm{rank}(y') \in \left\{1,2,3\right\}$ and it can vary from point to point. For the unbounded one-sided Sierpi\'nski gasket, every vertex other than the origin has rank $2$, and so the situation is `homogeneous' and much simpler than the general one.
This also shows that our extension of the construction in \cite{bib:KPP-PTRF} to the general case of planar nested fractals is non-trivial and requires a substantial improvement of the previous argument.

The proof of the continuity of the functions $g_M(t,x,y)$ (Lemma \ref{lem:properties}) requires a careful analysis of the rank of vertices. However, the main difficulty occurs in the proof of Theorem \ref{thm:transition}, which is absolutely fundamental for our investigations and further applications, also outside of this paper. To overcome this difficulty, we have to track the joint distribution of the consecutive hitting times of the fractal $M$-grid for the `free' process and the labels of the vertices attained by the process at these hitting times (Lemma \ref{lem:lawseq}(1)). This is based on a delicate induction procedure. Another difficulty which arose while proving the various regularity properties of the densities $g_M(t,x,y)$ is of analytic type. In the case of Sierpi\'nski gasket, various integral estimates needed in proving the boundedness, continuity and symmetry properties of such functions were based on the property that any $m$-complex $\mathcal{K}^{\left\langle m \right\rangle}$ agrees with the Euclidean ball $B(0,2^m)$ restricted to the fractal and that the geodesic metric is Lipschitz equivalent with the Euclidean distance. In the general case of simple nested fractals this is no longer true (it even might happen that the geodesic metric cannot be defined at all!). To overcome this obstacle, we use a new idea which is based on an application of the graph metric (Appendix \ref{sec:app}). This approach works well in Lemmas \ref{lem:properties} and \ref{lem:symmetry}.

The paper is organized as follows. In Section 2 we collect essentials on the constructions
of planar simple nested fractals and definitions of related objects, and we prove the basic
geometric result in Proposition \ref{pro:plane}. We also introduce the definition of the graph metric.
In consecutive  subsections of Section 3 we introduce and discuss the concept of GLP and give the sharp
sufficient condition for it to hold (Proposition \ref{pro:glp}). We also define and discuss the properties
(Proposition \ref{thm:composition}) of the `folding' projections and give several direct-to-check sufficient conditions
for the GLP. In the case of fractals with even number of essential fixed points, we also give a full characterization of this
property (Theorem \ref{th:even_ess}). In Section 4 we recall the basic properties of the Brownian motion on an unbounded
simple nested fractal and define and prove further properties of the relected Brownian motion. The proof of our main
Theorem \ref{thm:main1} is postponed till the end this section and is preceded by a sequence of auxiliary lemmas.
The reader interested mostly in probabilistic development can skip the material of Section \ref{sec:labelling}
other than the definitions 
and pass directly to Section \ref{sec:reflected}.
In the last section, Appendix \ref{sec:app}, we prove the comparability of the graph metric and the
Euclidean distance (Lemma \ref{lem:metrics}) and give several related results.

\section{Unbounded simple nested fractals} \label{sec:usnf}

The introductory part of this section follows the exposition of \cite{bib:Lin,bib:kpp-sausage,bib:kpp-sto}. Consider a collection of similitudes $\Psi_i : \mathbb{R}^2 \to \mathbb{R}^2$ with a common scaling factor $L>1,$  and a common isometry part $U,$ i.e. $\Psi_i(x) = (1/L) U(x) + \nu_i,$  where  $\nu_i \in \mathbb{R}^2$, $i \in \{1, ..., N\}.$ We shall assume $\nu_1 = 0$.
There exists a unique nonempty compact set $\mathcal{K}^{\left\langle 0\right\rangle}$ (called {\em the  fractal generated by the system} $(\Psi_i)_{i=1}^N$) such that $\mathcal{K}^{\left\langle 0\right\rangle} = \bigcup_{i=1}^{N} \Psi_i\left(\mathcal{K}^{\left\langle 0\right\rangle}\right)$.  As $L>1$, each similitude has exactly one fixed point and there are exactly $N$ fixed points of the transformations $\Psi_1, ..., \Psi_N$.

\begin{definition}[\textbf{Essential fixed points}]
A fixed point $x \in \mathcal{K}^{\left\langle 0\right\rangle}$ is an essential fixed point if there exists another fixed point $y \in \mathcal{K}^{\left\langle 0\right\rangle}$ and two different similitudes $\Psi_i$, $\Psi_j$ such that $\Psi_i(x)=\Psi_j(y)$.
The set of all essential fixed points for transformations $\Psi_1, ..., \Psi_N$ is denoted by $V_{0}^{\left\langle 0\right\rangle}$, let $k=\# V^{\left\langle 0\right\rangle}_{0}$.
\end{definition}

\begin{example}
The Vicsek fractal (Figure \ref{fig:essfix}) is constructed by 5 similitudes, four of them map the fractal onto complexes in the corners (let us denote them $\Psi_1, \Psi_2, \Psi_3, \Psi_4$) while $\Psi_5$ maps it onto the central complex. In this case the isometry $U$ is just the identity. The fixed points $v_i$ of the $\Psi_i'$s for $1\leq i \leq 4$ are essential fixed points. For example, the vertex $v_1$ is an essential fixed point, because $\Psi_5(v_1) = \Psi_1(v_3)$.
On the other hand, the fixed point of $\Psi_5$ (inside the central complex) is mapped onto points inside the complexes which do not coincide with the images of other vertices by any similitudes.
\end{example}

\begin{figure}[ht]
\centering
	\includegraphics[scale=1]{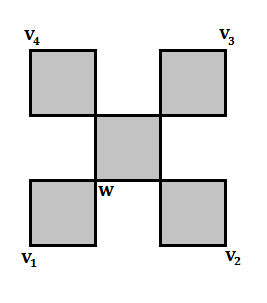}
\caption{Essential fixed points of the Vicsek fractal.}
\label{fig:essfix}
\end{figure}

The essential fixed points determine the general shape of complexes. In the example above the essential fixed points are the vertices of a square and each image of that square by some $\Psi_i$ (as in Figure \ref{fig:essfix}) contains a smaller copy of the fractal.

\begin{definition}[\textbf{Simple nested fractal}]
\label{def:snf}
 The fractal $\mathcal{K}^{\left\langle 0 \right\rangle}$ generated by the system $(\Psi_i)_{i=1}^N$ is called a \emph{simple nested fractal (SNF)} if the following conditions are met.
\begin{enumerate}
\item $\# V_{0}^{\left\langle 0\right\rangle} \geq 2.$
\item (Open Set Condition) There exists an open set $U \subset \mathbb{R}^2$ such that for $i\neq j$ one has\linebreak $\Psi_i (U) \cap \Psi_j (U)= \emptyset$ and $\bigcup_{i=1}^N \Psi_i (U) \subseteq U$.
\item (Nesting) $\Psi_i\left(\mathcal{K}^{\left\langle 0 \right\rangle}\right) \cap \Psi_j \left(\mathcal{K}^{\left\langle 0 \right\rangle}\right) = \Psi_i \left(V_{0}^{\left\langle 0\right\rangle}\right) \cap \Psi_j \left(V_{0}^{\left\langle 0\right\rangle}\right)$ for $i \neq j$.
\item (Symmetry) For $x,y \in V_{0}^{\left\langle 0\right\rangle},$ let $S_{x,y}$ denote the symmetry with respect to the line bisecting the segment $\left[x,y\right]$. Then
\begin{equation*}
\forall i \in \{1,...,M\} \ \forall x,y \in V_{0}^{\left\langle 0\right\rangle} \ \exists j \in \{1,...,M\} \ S_{x,y} \left( \Psi_i \left(V_{0}^{\left\langle 0\right\rangle} \right) \right) = \Psi_j \left(V_{0}^{\left\langle 0\right\rangle} \right).
\end{equation*}
\item (Connectivity) On the set $V_{-1}^{\left\langle 0\right\rangle}:= \bigcup_i \Psi_i \left(V_{0}^{\left\langle 0\right\rangle}\right)$ we define graph structure $E_{-1}$ as follows:\\
$(x,y) \in E_{-1}$ if and only if $x, y \in \Psi_i\left(\mathcal{K}^{\left\langle 0 \right\rangle}\right)$ for some $i$.\\
Then the graph $(V_{-1}^{\left\langle 0\right\rangle},E_{-1} )$ is required to be connected.
\end{enumerate}
\end{definition}

If  $\mathcal{K}^{\left\langle 0 \right\rangle}$ is a simple nested fractal, then we let
\begin{align} \label{eq:Kn}
\mathcal{K}^{\left\langle M\right\rangle} = L^M \mathcal{K}^{\left\langle 0\right\rangle}, \quad M \in \mathbb{Z},
\end{align}
and
\begin{align} \label{eq:Kinfty}
\mathcal{K}^{\left\langle \infty \right\rangle} = \bigcup_{M=0}^{\infty} \mathcal{K}^{\left\langle M\right\rangle}.
\end{align}
The set $\mathcal{K}^{\left\langle \infty \right\rangle}$ is the \textbf{unbounded simple nested fractal (USNF)} we shall be working with (see \cite{bib:kpp-sausage}).  Its fractal (Hausdorff) dimension is equal to  $d_f=\frac{\log N}{\log L}$. The Hausdorff  measure in dimension $d_f$ will be denoted by $\mu$. It will be normalized to have $\mu\left(\mathcal{K}^{\left\langle 0\right\rangle}\right)=1$. It serves as a `uniform' measure on $\mathcal{K}^{\left\langle \infty \right\rangle}.$

The remaining notions are collected in a single definition.
\begin{definition} Let $M\in\mathbb Z.$
\begin{itemize}
\item[(1)] $M$-complex: \label{def:Mcomplex}
every set $\Delta_M \subset \mathcal{K}^{\left\langle \infty \right\rangle}$ of the form
\begin{equation} \label{eq:Mcompl}
\Delta_M  = \mathcal{K}^{\left\langle M \right\rangle} + \nu_{\Delta_M},
\end{equation}
where $\nu_{\Delta_M}=\sum_{j=M+1}^{J} L^{j} \nu_{i_j},$ for some $J \geq M+1$, $\nu_{i_j} \in \left\{\nu_1,...,\nu_N\right\}$, is called an \emph{$M$-complex}.
\item[(2)] Vertices of the $M-$complex \eqref{eq:Mcompl}: the set $V\left(\Delta_M\right) =L^MV_0^{\langle 0 \rangle}+\nu_{\Delta_M}= L^{M} V^{\left\langle 0 \right\rangle}_0 + \sum_{j=M+1}^{J} L^{j} \nu_{i_j}$.
\item[(3)] Vertices of $\mathcal{K}^{\left\langle M \right\rangle}$:
$$
V^{\left\langle M\right\rangle}_{M} = V\left(\mathcal{K}^{\left\langle M \right\rangle}\right) = L^M V^{\left\langle 0\right\rangle}_{0}.
$$
\item[(4)] Vertices of all $M$-complexes inside a $(M+m)$-complex for $m>0$:
$$
V_M^{\langle M+m\rangle}= \bigcup_{i=1}^{N} V_M^{\langle M+m-1\rangle} + L^M \nu_i.
$$
\item[(5)] Vertices of all 0-complexes inside the unbounded nested fractal:
$$
V^{\left\langle \infty \right\rangle}_{0} = \bigcup_{M=0}^{\infty} V^{\left\langle M\right\rangle}_{0}.
$$
\item[(6)] Vertices of $M$-complexes from the unbounded fractal:
$$
V^{\left\langle \infty \right\rangle}_{M} = L^{M} V^{\left\langle \infty \right\rangle}_{0}
$$
\item[(7)] The set of all $M$-complexes from $\mathcal{K}^{\left\langle \infty \right\rangle}:$  denoted by  $\mathcal{T}_M$.
\item[(8)] The unique $M$-complex containing $x,$  $x \in \mathcal{K}^{\left\langle \infty \right\rangle}\backslash V^{\left\langle \infty \right\rangle}_{M}$,  is denoted by $\Delta_M (x)$ .
\end{itemize}
\end{definition}

\begin{figure}[ht]\label{fig:lindstrom}
\centering
	\includegraphics[scale=0.04]{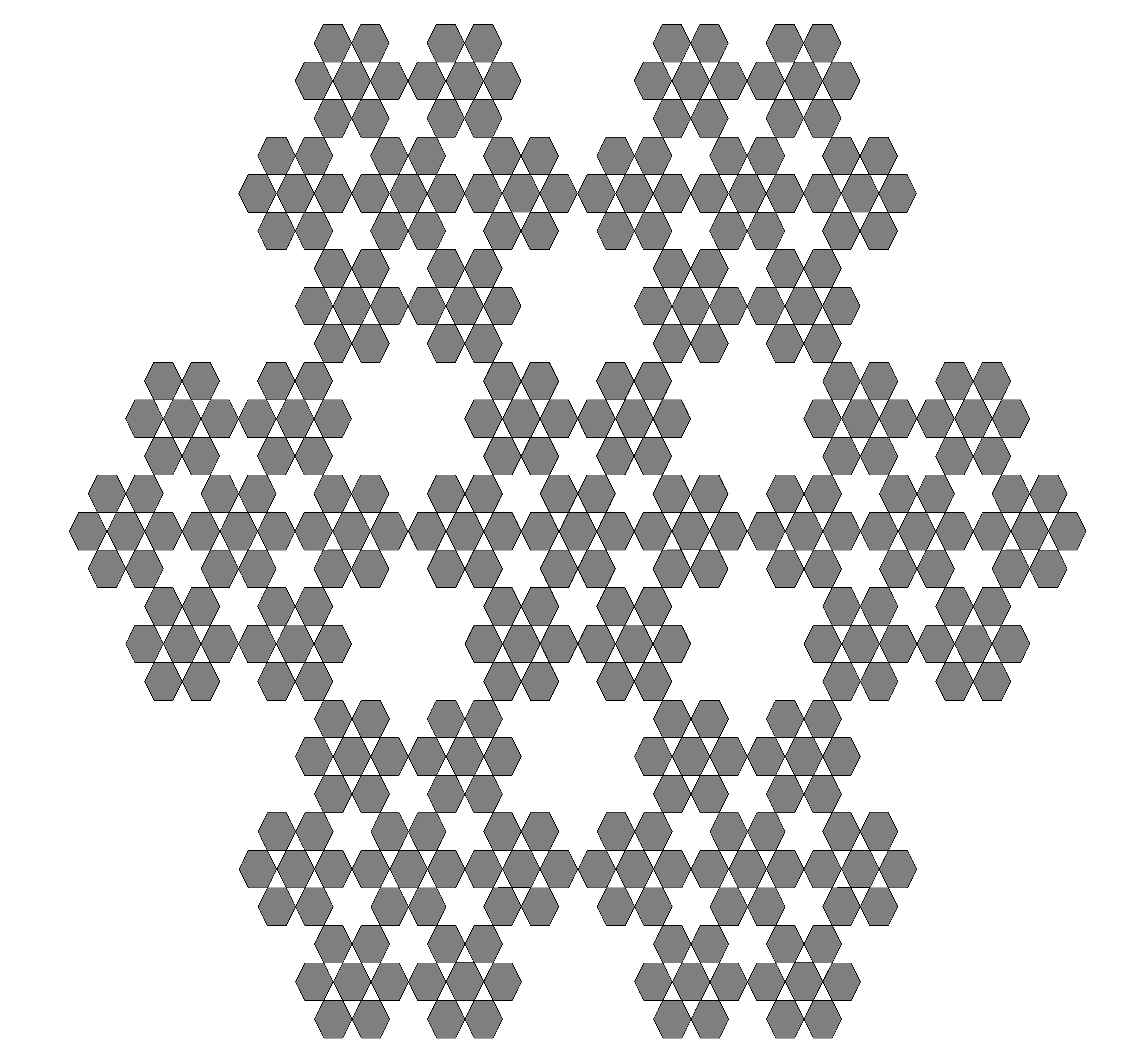}
\caption{An example of a nested fractal: the Lindstr\o m snowflake. It is constructed by 7 similitudes with $L=3$. It has 7 fixed points, but only 6 essential fixed points.}
\end{figure}

Below we need the following lemma.
\begin{lemma}{\cite[Lemma 3.14]{bib:Koch}}
\label{lem:koch}
Let $v \in V^{\left\langle 0 \right\rangle}_{0}$. Then there exist exactly one $i \in \{1,...,N\}$ such that $v \in \Psi_i \left( V^{\left\langle 0 \right\rangle}_{0}\right)$.
\end{lemma}

Building blocks of simple nested fractals (`complexes') are regular polygons. This was first conjectured in \cite{bib:Bar, bib:Kig} (see also \cite[Rem. 1.2]{bib:GI}). We use this fact below in an essential way, and so to make the paper self-contained, we provide a proof of this property, based on Lemma \ref{lem:koch}.

\begin{proposition} We have the following.
\label{pro:plane}
\begin{enumerate}
\item
If $k \geq 3$, then points from $V_{0}^{\left\langle 0\right\rangle}$ are the vertices of a regular polygon.
\item
If $k=2$, then $\mathcal{K}^{\left\langle 0 \right\rangle}$ is just a segment connecting $x_1$ and $x_2$.
\end{enumerate}
\end{proposition}
\begin{proof}
(1)
Let us denote the convex hull of $V_{0}^{\left\langle 0\right\rangle}$ by $\mathcal{H}_{0}^{\left\langle 0\right\rangle}$ and let $\mathcal{H}_{1}^{\left\langle 0\right\rangle} = \bigcup_{i=1}^{N} \Psi_i\left(\mathcal{H}_{0}^{\left\langle 0\right\rangle}\right)$. Then $\mathcal{H}_{0}^{\left\langle 0\right\rangle}$ is a polygon with vertices in some points of $V_{0}^{\left\langle 0\right\rangle}$. We will show that no vertex from $V_{0}^{\left\langle 0\right\rangle}$ lies in the interior of this figure.

According to the symmetry condition for nested fractals, no three essential fixed points are collinear. Also, for every $x_i,x_j \in V_{0}^{\left\langle 0\right\rangle}$ the line bisecting the segment $\left[x_i,x_j\right]$ is an axis of symmetry of $\mathcal{H}_{1}^{\left\langle 0\right\rangle}$. All axes of symmetry of a figure intersect at a single point $P$ -- its barycenter, and (since $\mathcal{H}_{0}^{\left\langle 0\right\rangle}$ is convex) $P$ lies inside $\mathcal{H}_{0}^{\left\langle 0\right\rangle}$. Moreover, if there were a point $x_i$ in $V_{0}^{\left\langle 0\right\rangle}$ lying in the interior of $\mathcal{H}_{0}^{\left\langle 0\right\rangle}$, then we could have picked a vertex $x_j$ of $\mathcal{H}_{0}^{\left\langle 0\right\rangle}$ making the angle $\angle x_j x_i P$  obtuse. Then the line bisecting $\left[x_i,x_j\right]$ would have been an axis of symmetry of $\mathcal{H}_{1}^{\left\langle 0\right\rangle}$ not containing the point $P,$ a contradiction. Therefore all vertices of $V_{0}^{\left\langle 0\right\rangle}$ are  vertices of the polygon $\mathcal{H}_{0}^{\left\langle 0\right\rangle}$.

Next, we  label the vertices $x_1, x_2, ..., x_n$ in such an order that segments $\left[x_1 x_2\right], \left[x_2 x_3\right], ..., \left[x_n x_1\right]$ are the edges of the polygon $\mathcal{H}_{0}^{\left\langle 0\right\rangle}.$ For simplicity  we can assume that $x_i$ is the fixed point of $\Psi_i$. Then the symmetry $S_{1,3}$ in the line bisecting the segment $\left[x_1 x_3\right]$ transports $x_1=\Psi_1(x_1)$ to $x_3=\Psi_3(x_3)$.
From Lemma \ref{lem:koch} we see that $\Psi_{1} \left(\mathcal{K}^{\left\langle 0 \right\rangle}\right)$ and $\Psi_{3} \left(\mathcal{K}^{\left\langle 0 \right\rangle}\right)$ are the only complexes containing $x_1$ and $x_3$ respectively.

The symmetry condition gives that the image of $\Psi_{1}\left(V_{0}^{\left\langle 0\right\rangle}\right)$ is $\Psi_{3}\left(V_{0}^{\left\langle 0\right\rangle}\right)$. As all the similitudes are based on a common isometry $U$, the images of $x_1,...,x_k$ by all the similitudes are either placed clockwise, or they all are placed counter-clockwise. In any case, the points $\Psi_1\left(x_2\right)$ and $\Psi_3\left(x_2\right)$ are located on the same side of the segment $[x_1,x_3]$.
Finally, as $S_{1,3} \left(\Psi_{1}\left(V_{0}^{\left\langle 0\right\rangle}\right)\right) = \Psi_{3}\left(V_{0}^{\left\langle 0\right\rangle}\right)$, $\Psi_1\left(x_2\right)$ is adjacent to $\Psi_1\left(x_1\right)$, and $\Psi_3\left(x_2\right)$ is adjacent to $\Psi_3\left(x_3\right)$, we conclude that $S_{1,3} \left(\Psi_1\left(x_2\right)\right) = \Psi_3\left(x_2\right)$. In the next step we see that since the  segments $[\Psi_1\left(x_1\right), \Psi_1\left(x_2\right)]$ and $[\Psi_3\left(x_3\right), \Psi_3\left(x_2\right)]$ are symmetric to each other, they have  equal length, and  consequently  $\left|[x_1,x_2]\right|=\left|[x_2, x_3]\right|$ as well.

By repeating this reasoning we find that all edges of the polygon have the same length. Similarly, we can show that all angles in the polygon have the same measure.

The symmetry $S_{2,3}$ in the line bisecting the segment $[x_2,x_3]$ transports $x_2=\Psi_2(x_2)$ to $x_3=\Psi_3(x_3)$. Again, from the symmetry condition the image of $\Psi_{2}\left(V_{0}^{\left\langle 0\right\rangle}\right)$ is $\Psi_{3}\left(V_{0}^{\left\langle 0\right\rangle}\right)$.
Consequently, it follows that $S_{2,3} \left(\Psi_2\left(x_3\right)\right) = \Psi_3\left(x_2\right)$ and $S_{2,3} \left(\Psi_2\left(x_1\right)\right) = \Psi_3\left(x_4\right)$. In the final step we see that the equality of angles at the vertices $\Psi_2\left(x_2\right)$ and $\Psi_3\left(x_3\right)$ gives  the equality of angles $\angle x_1x_2x_3$ and $\angle x_2x_3x_4$.
The proof for $k\geq 3$ is completed.

(2) Let now $k=2.$ Then $\mathcal{H}_{0}^{\left\langle 0\right\rangle}$ is a line segment, and all
its images in the mappings $\Psi_i$ are parallel. Connectivity of the graph $(V_{-1}^{\left\langle 0\right\rangle},E_{-1} )$ implies that they are also parallel to $\mathcal{H}_{0}^{\left\langle 0 \right\rangle}$.
This means that $U$ is either the identity, or the symmetry in the line perpendicular to $\left[x_1,x_2\right]$.

Indeed, it is impossible to construct a polygonal chain connecting $x_1$ and $x_2$ using parallel segments which would not be parallel to the segment $[x_1, x_2]$. Therefore we have to rule out all isometries which are based on rotations (different than those by angle $\pi$ or $2\pi$).

The connectivity of the graph $(V_{-1}^{\left\langle 0\right\rangle},E_{-1} )$ and Lemma \ref{lem:koch} show that there are no points from $\mathcal{K}^{\left\langle 0 \right\rangle}$ outside of $[x_1,x_2]$ and that $U$ cannot be the symmetry in the line bisecting $[x_1, x_2]$. Using the connectivity again we see that $\mathcal{K}^{\left\langle 0 \right\rangle}$ is equal to $[x_1,x_2]$ (we reject Cantor-type sets).

\end{proof}
 Along the way we concluded that when $k=2,$ then  the isometry $U$ is the identity or a translation by some vector $\nu$. As we have previously assumed that $\nu_1 = 0$, in fact we have $U=\Id$.
Below we prove that this property is true for all simple nested fractals.

From now on we shall assume that $k \geq 3$, because for $k=2$ the fractal becomes trivial.

\begin{proposition}
For simple nested fractals the defining isometry $U$ is the identity mapping.
\end{proposition}
\begin{proof}
By Proposition \ref{pro:plane} we know that $\mathcal{H}_{1}^{\left\langle 0\right\rangle}$ is composed of regular polygons connected at their vertices.

Take two neighboring essential fixed points $x_1$, $x_2$ (as in the proof of Proposition \ref{pro:plane}). Let $i\neq j$ be indices such that $\Psi_i\left(x_1\right)=x_1$, $\Psi_j\left(x_2\right)=x_2$. Without loss of generality $i=1, j=2.$ Then the line bisecting the segment $\left[x_1, x_2\right]$ is an axis of symmetry of $\mathcal{H}_{1}^{\left\langle 0\right\rangle}$. The image of $\Psi_1\left(\mathcal{H}_{0}^{\left\langle 0\right\rangle}\right)$ in this symmetry is $\Psi_2\left(\mathcal{H}_{0}^{\left\langle 0\right\rangle}\right)$.   If $U$ were a rotation, or a symmetry in a line not parallel to the line bisecting $\left[x_1, x_2\right]$, then these two figures would have different alignment, giving  a contradiction. Remaining options are  that $U$ is the identity, or the symmetry in the line bisecting $\left[x_1, x_2\right]$ (additional translation is not permited since $\nu_1=0$).

Take now an essential fixed point $x_3,$ a neighbor of  $x_2,$ then copy the reasoning above for the segment $\left[x_2,x_3\right]$ to conclude that $U$ is the identity or the symmetry in the line bisecting $\left[x_2,x_3\right]$. As the lines bisecting $[x_1,x_2]$ and $[x_2,x_3]$ are not parallel, we conclude that $U=\Id.$
\end{proof}

We now introduce the '$M$-graph' distance on $\mathcal K^{\langle \infty\rangle} \times \mathcal K^{\langle \infty\rangle}$, which will be needed in the next section.

\begin{definition}\label{def:graph-distance}
 For $M \in \mathbb Z$ and $x,y \in \mathcal{K}^{\left\langle \infty \right\rangle}$ let
\begin{equation}
\label{def:graphmetric}
d_M (x,y):= \left\{ \begin{array}{ll}
0, & \textrm{if } x=y ;\\
1, & \textrm{if there exists } \Delta_M \in \mathcal{T}_M \textrm{ such that } x,y \in \Delta_M ;\\
n>1, & \textrm{if there does not exist } \Delta_M \in \mathcal{T}_M \textrm{ such that } x,y \in \Delta_M \textrm{ and } n \textrm{ is the smallest}\\
&  \textrm{number for which there exist } \Delta_M^{(1)}, \Delta_M^{(2)}, ..., \Delta_M^{(n)} \in \mathcal{T}_M \textrm{ such that } x \in \Delta_M^{(1)},\\
&  y \in \Delta_M^{(n)} \textrm{ and } \Delta_M^{(i)} \cap \Delta_M^{(i+1)} \neq \emptyset \textrm{ for  }1 \leq i \leq n-1.
\end{array} \right.
\end{equation}
Moreover, for a fixed $x \in \mathcal{K}^{\left\langle \infty \right\rangle}$ we define inductively the collection of $M-$complexes `lying at distance $n$ from a given point $x$':
\begin{equation}
\label{eq:xComplexdist}
\begin{split}
\mathcal{L}_{M,1,x} & = \left\{\Delta_M \in \mathcal{T}_M : x \in \Delta_M \right\};\\
\mathcal{L}_{M,n+1,x} & = \left\{\Delta_M \in \mathcal{T}_M \backslash \bigcup_{i=1}^{n} \mathcal{L}_{M,i,x} : \exists \widetilde{\Delta}_M \in \mathcal{L}_{M,n,x} \ \widetilde{\Delta}_M \cap \Delta_M \neq \emptyset \right\}, \quad n \geq 1.
\end{split}
\end{equation}
\end{definition}
Equivalently,
\begin{equation*}
\mathcal{L}_{M,n,x} = \left\{\Delta_M \in \mathcal{T}_M : \sup_{y\in \Delta_M} d_M (x,y)=n \right\}, \quad n \geq 1.
\end{equation*}
Further properties of the graph distance and the upper estimate for the cardinality of the families $\mathcal{L}_{M,n,x}$ are given in the Appendix \ref{sec:app}.

\section{Good labelling and projections}\label{sec:labelling}
In this section we present the notion of good labeling. Good labeling gives rise to the `folding' projection $\pi_M: \mathcal K^{\langle \infty\rangle}\to \mathcal K^{\langle M\rangle},$ and this projection will be in the next section used to define the reflected Brownian motion on $\mathcal K^{\langle M\rangle}.$

\subsection{The concept of good labelling of vertices} \label{subsec:glp}

  Consider the alphabet of $k$ symbols $\cA:=\left\{a_1, a_2,a_3,...,a_k\right\}$,  where  $ k=\# V^{\left\langle 0\right\rangle}_{0}\geq 3.$ The elements of $\cA$ are called labels.
 \begin{definition}\label{def:labeling}
  Let $M \in \mathbb{Z}$. A \emph{labelling function of order $M$} is any map $l_M: V^{\left\langle \infty \right\rangle}_{M} \to \cA.$
  \end{definition}

We now introduce the concept of good labelling of vertices of USNFs, which generalizes the labelling procedure proposed in \cite{bib:KPP-PTRF, bib:KaPP2} in the case of unbounded Sierpi\'nski triangle.
Recall that (Proposition \ref{pro:plane}) every $M$-complex $\Delta_M$  is a regular polygon with $k$ vertices. In particular, there exist exactly $k$ different rotations around the barycenter of $\mathcal K^{\langle M \rangle}$,  mapping $V^{\left\langle M \right\rangle}_{M}$ onto $V^{\left\langle M \right\rangle}_{M}.$ They will be denoted by  $\{R_1, ..., R_k\}=: \mathcal{R}_M $ (ordered in such a way that for $i=1,2,...,k,$ the rotation $R_i$ rotates by angle $\frac{2\pi i}{k})$.
\begin{definition}[\textbf{Good labelling function of order $M$}]
\label{def:glp} Let $M \in \mathbb{Z}$.  A function $\ell_M: V^{\left\langle \infty \right\rangle}_{M} \to \cA$  is called a \emph{good labelling function of order $M$} if the following conditions are met.
\begin{itemize}
\item[(1)] The restriction of $\ell_M$ to $V^{\left\langle M \right\rangle}_{M}$ is a bijection onto $\cA$.
\item[(2)] For every $M$-complex $\Delta_M$ represented as
$$
\Delta_M  = \mathcal{K}^{\left\langle M \right\rangle} +\nu_{\Delta_M},
$$
where $\nu_{\Delta_M}=  \sum_{j=M+1}^{J} L^{j} \nu_{i_j},$  with some $J \geq M+1$ and $\nu_{i_j} \in \left\{\nu_1,...,\nu_N\right\}$ (cf. Def. \ref{def:Mcomplex}), there exists a rotation $R_{\Delta_M} \in \mathcal{R}_M$ such that
\begin{align}\label{eq:rot}
\ell_M(v)=\ell_M\left(R_{\Delta_M}\left(v -\nu_{\Delta_M}\right)\right) , \quad v \in V\left(\Delta_M\right).
\end{align}
\end{itemize}
\end{definition}
 An USNF $\mathcal{K}^{\left\langle \infty \right\rangle}$ is said to have the  \emph{good labelling property of order $M$} if a good labelling function of order $M$ exists. Note that, in fact, for every $M$-complex $\Delta_M$  the restriction of a good labelling function to $V\left(\Delta_M\right)$ is a bijection onto $\cA$.

\bigskip

The good labeling property means that the rotation of $R_{\Delta_M}$ applied to $\Delta_M-\nu_{\Delta_M}$ maps the vertices of $\Delta_M$ onto vertices of $V_M^{\left\langle M\right\rangle}$ with matching labels.

Thanks to the selfsimilar structure of $\mathcal{K}^{\left\langle \infty \right\rangle},$ the \emph{good labelling property of order $M$} for some $M \in \mathbb{Z}$ is equivalent to this property of any other order $\widetilde{M} \in \mathbb{Z}$. This gives rise to the following general definition.

\begin{definition}[\textbf{Good labelling property}] \label{def:glp_gen} An USNF $\mathcal{K}^{\left\langle \infty \right\rangle}$ is said to have the \emph{good labelling property (GLP in short)} if it has the {good labelling property of order $M$} for some $M \in \mathbb{Z}$.
\end{definition}

\begin{remark} \label{rem:rem_glp}
\noindent {\rm
If $l_{M}$ is a good labelling function of order $M$ on $V_M^{\langle\infty\rangle}$, then it is typically not true that restricting $l_M$ to $V^{\left\langle \infty \right\rangle}_{M+1}$  gives  GLP of order $M+1$ (see Figure \ref{fig:levels}). This contrasts the case of the unbounded Sierpi\'nski gasket (see e.g. \cite{bib:KaPP2}), where a good labelling function of order $M$ automatically provided GLP of every order $\widetilde M \geq M$.

\begin{figure}[ht]
\centering
	\includegraphics[scale=0.3]{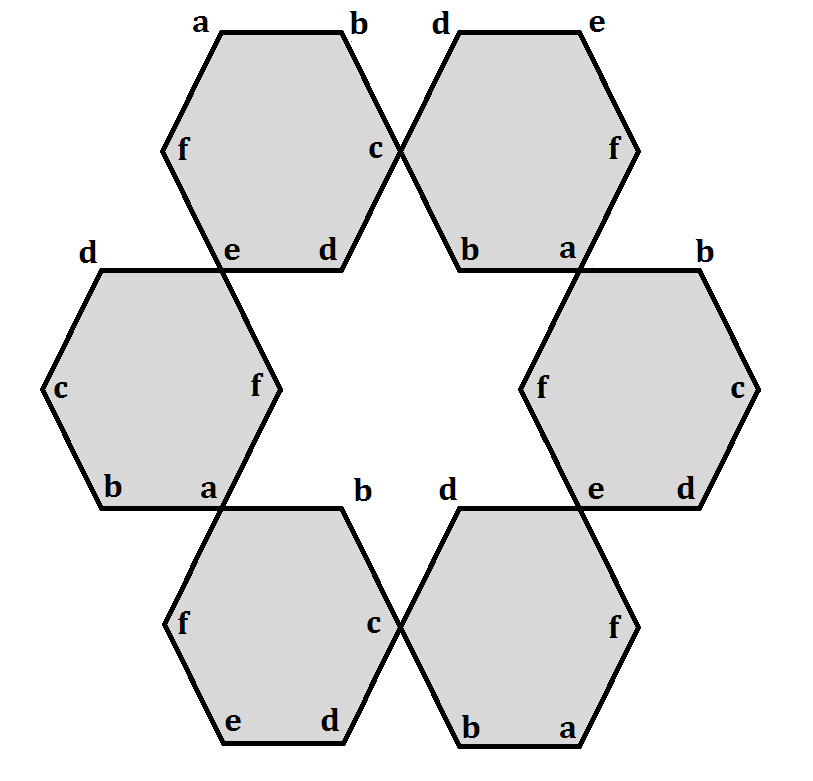}
\caption{Labelling of vertices of order $M$ with labels $a,b,c,d,e,f$ inside a hexagonal $(M+1)$-complex. The $(M+1)$-complex has two vertices with labels $a$, $c$, $e$ and none with $b$, $d$ or $f$.}
\label{fig:levels}
\end{figure}
}
\end{remark}

\begin{proposition} \label{pro:uniquelabel}
For USNF's with the GLP, for any $M\in \mathbb Z$ the good labeling of order $M$ is unique up to a permutation of the alphabet set $\cA.$ In particular,
if $\mathcal{K}^{\left\langle \infty \right\rangle}$ has GLP and
a bijection $\widetilde{\ell}_M: V^{\left\langle M \right\rangle}_{M} \to \mathcal{A}$ is given, then there exists a unique
good labeling function $\ell_M:V_M^{\langle \infty\rangle} \to \cA$ such that $\ell_M|_{V_{M}^{\langle M\rangle}}=\widetilde \ell_M.$
\end{proposition}


In other words, with the labelling of vertices from $V^{\left\langle M \right\rangle}_{M}$ given, there is exactly one way to label all other vertices from $V^{\left\langle \infty \right\rangle}_{M}$ in the way providing the GLP.

\begin{proof}
Suppose that $\ell_M,\ell_M': V_M^{\langle\infty\rangle}\to\mathcal A$ are two
good labeling functions. By definition, $\ell_M, \ell'_M$ restricted to $V_M^{\langle M\rangle}$
are bijections onto  the alphabet set $\mathcal A.$ Therefore there is a
permutation $\sigma:\mathcal A\to\mathcal A$ such that $\sigma\circ\ell_M'|_{V_M^{\langle M\rangle}}=
\ell_M|_{V_M^{\langle M\rangle}}.$ We claim that $\sigma\circ\ell'_M=\ell_M.$

Indeed, any good labeling of $V_M^{\langle\infty\rangle}$ is determined by its values on $V_M^{\langle M\rangle}.$
To see this, suppose that the labeling on $V_M^{\langle M\rangle}$ is given.  Each of the $M-$complexes neighbor to $\mathcal K^{\langle M\rangle}$ has exactly one vertex common with $\mathcal{K}^{\left\langle M \right\rangle},$  and it already has a label. If we are to preserve the orientation of labels (which is the essence of the good labeling), there is just one way to put labels on all  other vertices of these complexes.
Then, recursively, in the $(n+1)$-th step we label vertices of all complexes neighboring the complexes labelled in $n$-th step that has  not been labelled yet. It can be done uniquely.

As it is clear that $\sigma\circ\ell'_M$ is a good labeling function, agreeing with $\ell_M$ on $V_M^{\langle M\rangle},$ then the argument above shows that they do agree on $V_M^{\langle \infty \rangle}.$
\end{proof}

Below we present a sufficient and necessary condition for the GLP to hold. It will serve as a tool to determine the GLP in specific cases in Section \ref{sec:sufficient}.
\begin{proposition}
\label{pro:glp}
Let $\mathcal{K}^{\left\langle \infty \right\rangle}$ be a planar USNF, $M\in \mathbb Z$ and let ${\ell}_{M,0}: V^{\left\langle M \right\rangle}_{M} \to \mathcal{A}$ be a bijection. Then $\mathcal{K}^{\left\langle \infty \right\rangle}$ has the GLP if and only if there exists an extension of ${\ell}_{M,0}$ to $\widetilde{\ell}_{M,0}: V^{\left\langle M+1 \right\rangle}_{M} \to \mathcal{A}$ such that for every $M$-complex $\Delta_M \subset \mathcal{K}^{\left\langle M+1 \right\rangle}$ represented as (cf. Definition \ref{def:Mcomplex})
\begin{equation*}
\Delta_M = \mathcal{K}^{\left\langle M \right\rangle} + L^{M+1} \nu_{i_{M+1}}, \quad \nu_{i_{M+1}} \in \left\{\nu_1,...,\nu_N\right\},
\end{equation*}
there exists a rotation $R_{\Delta_M} \in \mathcal{R}_M$ such that
\begin{equation} \label{eq:aux_agr}
\widetilde{\ell}_{M,0}\left(R_{\Delta_M}\left(v - L^{M+1} \nu_{i_{M+1}}\right)\right) = \widetilde{\ell}_{M,0}(v), \quad v \in V\left(\Delta_M\right).
\end{equation}
\end{proposition}
This proposition means that if a labeling on $V_M^{\langle M\rangle}$ can be extended in a `good' way to $V_M^{\langle M+1\rangle},$
then it can be extended as a good labeling also to $V_M^{\langle\infty\rangle}.$


\begin{proof}
Let $\ell_{M,0}$ and $\widetilde{\ell}_{M,0}$ be as in the assumptions. We are going to construct a good labeling function
$\ell_M:V_M^{\langle \infty\rangle}\to\mathcal{A}.$

For  $v \in V^{\left\langle M+1 \right\rangle}_{M},$  define $\ell_M (v) = \widetilde{\ell}_{M,0} (v).$ Then we proceed recursively.
Suppose $\ell_M$ has been defined on $V_M^{\langle M+m\rangle},$ for some $m\geq 1.$ We shall put labels $\ell_M$  on $V^{\left\langle M+m+1 \right\rangle}_{M} \backslash V^{\left\langle M+m \right\rangle}_{M}.$
Observe that if $v \in V^{\left\langle M+m+1 \right\rangle}_{M+m},$ then $\frac{v}{L^m}\in V_{M}^{\langle M+1\rangle}.$
We define an auxiliary function $\kappa_{M+m}: V_{M+m}^{\langle M+m+1\rangle} \to \mathcal{A}$ by
\begin{equation}
\kappa_{M+m} (v) = \widetilde{\ell}_{M,0} \left(\frac{v}{L^m}\right), \quad v \in V^{\left\langle M+m+1 \right\rangle}_{M+m}.
\end{equation}
Then for each $(M+m)$-complex $ \mathcal{K}^{\left\langle M+m+1 \right\rangle}\supset\Delta_{M+m} = \mathcal{K}^{\left\langle M+m \right\rangle} + L^{M+m+1} \nu_{i_{M+m+1}}$ there exists a rotation $R_{\Delta_{M+m}} \in \mathcal{R}_{M+m}$ such that
\begin{equation} \label{eq:rot_id}
\kappa_{M+m}\left(R_{\Delta_{M+m}}\left(v - L^{M+m+1} \nu_{i_{M+m+1}}\right)\right) = \kappa_{M+m} (v), \quad v \in V\left(\Delta_{M+m}\right).
\end{equation}
This is a direct consequence of \eqref{eq:aux_agr} and the scaling property of the fractal (the set $V^{\left\langle M+m+1 \right\rangle}_{M+m}$ is
just a scaled-up version of $V^{\left\langle M+1 \right\rangle}_{M}$).

Now, once the rotations in \eqref{eq:rot_id} are identified, we may define $\ell_M$ on $V^{\left\langle M+m+1 \right\rangle}_{M}$ as follows
\begin{multline}
\label{eq:expansion}
\ell_M\left( v \right) = \ell_M\left(R_{\Delta_{M+m}}\left(v - L^{M+m+1} \nu_{i_{M+m+1}}\right)\right),\\
v \in V^{\left\langle M+m+1 \right\rangle}_{M}, \quad v \in \Delta_{M+m} = \mathcal{K}^{\left\langle M+m \right\rangle} + L^{M+m+1} \nu_{i_{M+m+1}}.
\end{multline}
In this way the function $\ell_{M}$ extends inductively to $\bigcup_{m=0}^{\infty} V^{\left\langle M+m \right\rangle}_{M} = V^{\left\langle \infty \right\rangle}_{M}$. Such an inductive procedure automatically gives that the condition (2) in Definition \ref{def:glp} holds true.

\end{proof}


Not every nested fractal has the GLP. An example is given below.

\begin{example}
\label{ex:snow}
One can label vertices of complexes of the Sierpi\'nski hexagon (Figure \ref{fig:legalhex}), but adding the central complex to form the Lindstr\o m snowflake (Figure \ref{fig:illegalsnow}) makes the labeling impossible.

Indeed, having labelled vertices of the bottom left complex clockwise as $a$, $b$, $c$, $d$, $e$, $f$ we se that the bottom right complex must have its left vertex labelled as $c$. Labelling other vertices of this complex clockwise determines that the label of the top left vertex is $d$. On the other hand, the middle complex has the bottom left vertex labelled as $b$, and so its bottom right vertex should be labelled as $a$. As a vertex cannot have two different labels, this fractal does not have the GLP.

\begin{figure}[ht]
\centering
\includegraphics[scale=0.25]{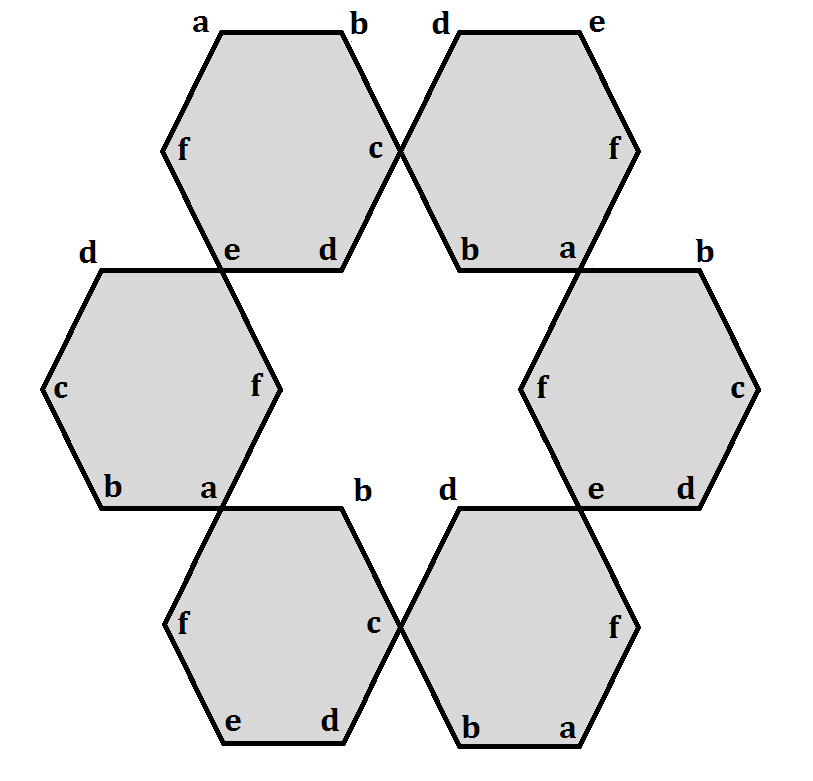}
\caption{The labeling of vertices of complexes of the Sierpi\'nski hexagon.}
\label{fig:legalhex}
\end{figure}

\begin{figure}[ht]
\centering
\includegraphics[scale=0.28]{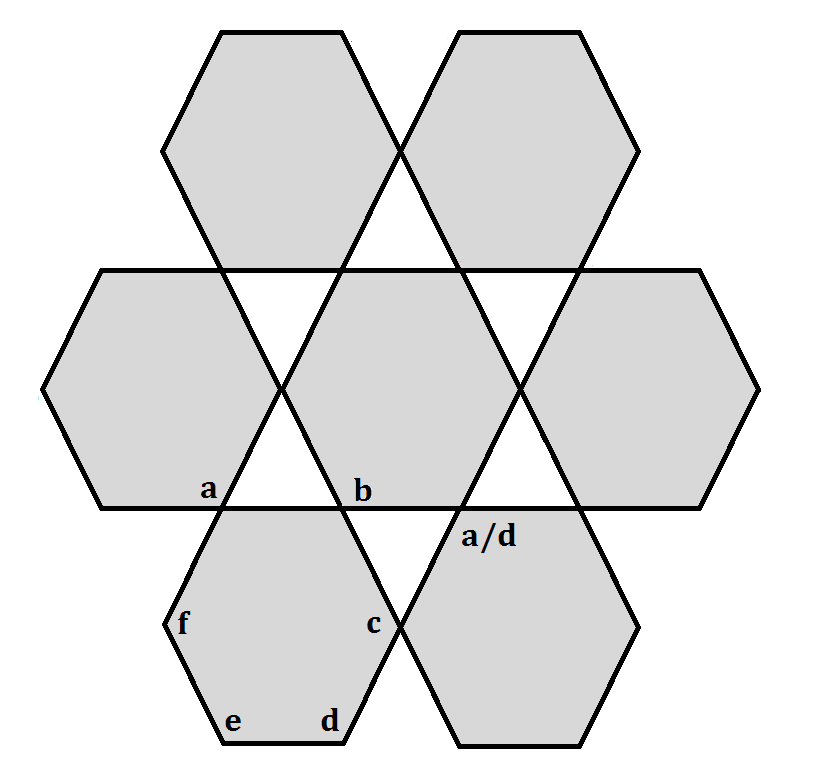}
\caption{Illegal labeling of vertices of complexes of the Lindstr\o m snowflake.}
\label{fig:illegalsnow}
\end{figure}

\end{example}

\subsection{Projections of planar USNFs and their properties} \label{subsec:proj}

For an unbounded fractal $\mathcal{K}^{\left\langle \infty \right\rangle}$ satisfying  the GLP, we
define a projection map $\pi_{M}$ from $\mathcal{K}^{\left\langle \infty \right\rangle}$  onto the primary $M$-complex $\mathcal{K}^{\left\langle M \right\rangle}$ by the formula
\begin{equation}
\pi_M(x) := R_{\Delta_M}\left(x -\nu_{\Delta_M}\right),\quad x\in \mathcal K^{\langle \infty\rangle},
\end{equation}
where $\Delta_M = \mathcal{K}^{\left\langle M \right\rangle} + \nu_{\Delta_M} = \mathcal{K}^{\left\langle M \right\rangle} + \sum_{j=M+1}^{J} L^{j} \nu_{i_j}$ is an $M$-complex containing $x$ and $R_{\Delta_M}\in\mathcal{R}_M$ is the unique rotation 
determined by \eqref{eq:rot}. More precisely,
\begin{itemize}
\item[(1)] if $x \notin V_M^{\langle \infty\rangle}$, then we take $\Delta_M = \Delta_M(x)$ (i.e. $\Delta_M$ is the unique $M$-complex containing $x$),
\item[(2)] if $x \in V_M^{\langle \infty\rangle}$, then $\Delta_M$ can be chosen as any of the $M$-complexes meeting at $x$.
\end{itemize}
\noindent
If $x$ is a vertex from $V_M^{\langle \infty\rangle},$  possibly belonging to more than one $M$-complex, then we can choose any of those complexes in the above definition -- thanks to the GLP the image does not depend on a particular choice of an $M$-complex containing $x$.

This projection restricted to any $M$-complex $\Delta_M$ is a bijection, therefore the inverse of this restriction,
$(\pi_M|_{\Delta_M})^{-1}=:\widetilde{\pi}_{\Delta_M},$  is well defined and given by the formula
\begin{equation*}
\widetilde{\pi}_{\Delta_M}(x) = R_{\Delta_M}^{-1}(x) + \nu_{\Delta_M},\quad x\in \mathcal K^{\langle M\rangle}
\end{equation*}
where $\Delta_M = \mathcal{K}^{\left\langle M \right\rangle} + \nu_{\Delta_M} = \mathcal{K}^{\left\langle M \right\rangle} + \sum_{j=M+1}^{J} L^{j} \nu_{i_j}$.

We can also project onto any other arbitrarily chosen $M-$complex $\Delta_M.$
\begin{definition}[\textbf{Projection onto an $M$-complex}]
\label{def:projection} Let $\Delta_M\in\mathcal T_M$ be fixed. Define \linebreak$\pi_{\Delta_M}:\mathcal{K}^{\left\langle \infty \right\rangle}\to\Delta_M$ by setting
\begin{equation}
\pi_{\Delta_M}(x) = \widetilde{\pi}_{\Delta_M}\left(\pi_{M}(x)\right).
\end{equation}
\end{definition}
Clearly, $\pi_{\mathcal{K}^{\left\langle M \right\rangle}} = \pi_{M}$, because $\widetilde{\pi}_{\mathcal{K}^{\left\langle M \right\rangle}} = \Id$.

\begin{remark}
{\rm Our definition of $\pi_M$ generalizes that in \cite{bib:KPP-PTRF}, where the case of the planar unbounded Sierpi\'{n}ski gasket was studied. In that paper, it was used that any $x(\in \Delta_M(x))$ can be uniquely represented
as a convex combination of vertices from $V(\Delta_M(x)),$ i.e.
$$
x = \sum_{i=1}^{k} x_i \cdot v_i(x),
$$
where $v_i(x)$ are vertices of $\Delta_M(x)$ and $x_i \in [0,1]$ satisfy $\sum_{i=1}^{k}x_i =1$.
For general nested fractals, this approach may fail. First, if $k>3$, then the above representation of $x$ may not be unique.
Second, in general case, an $M$-complex needs not be included in the convex hull of its vertices. The example of such a situation is given below (Figure \ref{fig:notconvex}). }

\begin{figure}[ht]
\centering
\includegraphics[scale=0.7]{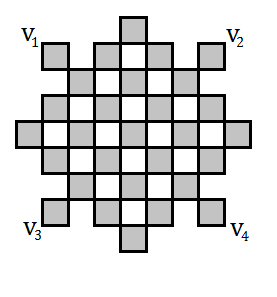}
\caption{The first step of the construction of a complex in a nested fractal which is not a subset of the convex hull of its four vertices. Here we have $k=4$ (complexes are squares and $V_0^{\langle 0\rangle} = \left\{ v_1, v_2,v_3,v_4\right\}$), $L=7$, $N=29$,  In the next iteration we replace each gray square with smaller copies of the whole figure.}
\label{fig:notconvex}
\end{figure}
\end{remark}

The next result states that the compositions of the two projection maps on different levels commute and are consistent with the projection on the finer level. It is important for our further applications. Its proof is a direct consequence of the GLP and it is omitted.

\begin{proposition}
\label{thm:composition}
If an USNF $\mathcal K^{\langle\infty\rangle}$ has the GLP, $M <\widetilde{M},$ and $\Delta_{M} \subset \Delta_{\widetilde{M}}$ are two complexes, then
$$
\pi_{\Delta_M}\circ \pi_{\Delta_{\widetilde {M}}} =
\pi_{\Delta_{\widetilde {M}}}\circ \pi_{\Delta_M}=\pi_{\Delta_M}.
$$
\end{proposition}

\bigskip

\subsection{Sufficient conditions for GLP of planar USNFs}\label{sec:sufficient} \label{subsec:suf_cond_glp}

In this section we will give the general geometric sufficient conditions for the good labelling property (cf. Definition \ref{pro:glp}) under which the projections $\pi_M$ can be properly defined. In other words, we will find and describe  general subclasses of simple nested fractals for which 
the projected processes can be well-defined.

We will analyze, on which unbounded nested fractals, given labeling of the vertices from $V^{\left\langle M\right\rangle}_{M}$, $M \in \mathbb{Z},$ we can label all other vertices from $V^{\left\langle \infty \right\rangle}_{M}$ in a unique way, such that the orientation of labels on each $M$-complex is preserved. In order to simplify the reasoning we will write proofs for $M=0$.


Recall that by $N$ we have denoted the number of similitudes generating $\mathcal{K}^{\left\langle 0 \right\rangle}$ and by $k$ the number of their essential fixed points, i.e., $k=\# V^{\left\langle 0\right\rangle}_{0}$. Throughout we always assume that $k \geq 3$. Our first result states that if the complexes are composed of triangles (i.e. $k=3$), then the GLP always holds. In its proof we use a labelling technique adapted from the papers \cite{bib:KPP-PTRF, bib:KaPP2}, where the Sierpi\'nski Gasket was studied. Note that if $k=3$, then $V_{0}^{\left\langle 0\right\rangle}$ is a subset of a lattice on the plane.

\begin{theorem} \label{th:triangles}
If $k=3$,
then $\mathcal{K}^{\left\langle \infty \right\rangle}$ has the GLP.
\end{theorem}
\begin{proof}
If there are three essential fixed points, then the vertices of a complex form an equilateral triangle. Without losing generality, we can assume that
\begin{equation}
V^{\left\langle 0\right\rangle}_{0} = \left\{ \left(0,0\right), \left(1,0\right), \left(\frac{1}{2},\frac{\sqrt{3}}{2}\right) \right\}
\end{equation}
and set
\begin{equation*}
\ell_0\left(\left(0,0\right)\right) = a, \quad \ell_0\left(\left(1,0\right)\right) = b, \quad \ell_0\left(\left(\frac{1}{2},\frac{\sqrt{3}}{2}\right)\right) = c.
\end{equation*}

We observe that $V^{\left\langle \infty \right\rangle}_{0} \subset \mathbb{Z}e_1 + \mathbb{Z}e_2$, where $e_1 = \left(1,0\right)$, $e_2 = \left(\frac{1}{2},\frac{\sqrt{3}}{2}\right)$.
Similarly to the labelling of vertices of the Sierpi\'nski Gasket in \cite{bib:KPP-PTRF}, we can represent every vertex $v \in V^{\left\langle \infty \right\rangle}_{0}$ as $v = n_1 e_1 + n_2 e_2, n_1,n_2 \in \mathbb{N}$ and this representation is unique.

We consider the commutative group of rotations (subgroup of all permutations) of labels: $\mathbb{A}_3 = \left\{ \Id, \left(a,b,c\right), \left(a,c,b\right)\right\}$ and denote $p_1= \left(a,b,c\right)$, $p_2 = \left(a,c,b\right)$. Clearly $p_1^3=\Id$, $p_2^3=\Id$.

We define $\ell_0$ on $V^{\left\langle \infty \right\rangle}_{0}$ as follows:
\begin{equation}
\ell_0\left( n_1 e_1 + n_2 e_2 \right) = \left(p_1^{n_1} \circ p_2^{n_2} \right)\left(a\right).
\end{equation}

By such labelling each $0$-complex of a form $\Delta_0  = \mathcal{K}^{\left\langle 0 \right\rangle} + \sum_{j=1}^{J} L^{j} \nu_{i_j} =  \mathcal{K}^{\left\langle 0 \right\rangle} + n_1 e_1 + n_2 e_2$ has the complete set of three labels on its vertices and the corresponding rotation $R_{\Delta_0}\in\mathcal {R}_0$ is such that
\begin{equation*}
\ell_0 \left(R_{\Delta_0} \left(x \right) \right) = \left(p_2^{n_1} \circ p_1^{n_2} \right)\left(\ell_0 (x)\right), x \in V^{\left\langle 0 \right\rangle}_{0}
\end{equation*}

In other words, the rotation $R_{\Delta_0}$ rotates the labelled points by such angle that the labels are permuted according to $p_2^{n_1} \circ p_1^{n_2}$.

\end{proof}

\begin{example}
Figure \ref{fig:triangleex} shows well-labelled vertices from $V^{\left\langle 1 \right\rangle}_{0}$ for the fractal with $k=3$, $N=15$, $L=6$. With $\ell_0$ on $V^{\left\langle 0 \right\rangle}_{0}$ given (labelling of the vertices of bottom leftmost triangle), we can label all vertices from $V^{\left\langle 1 \right\rangle}_{0}$ in such a way that the orientation of labels on each $0$-complex is preserved.

Observe that in this example, all three vertices from  $V^{\left\langle 1 \right\rangle}_{1}$ are labelled as $a$, so the labelling function $\ell_1$ on  $V^{\left\langle \infty \right\rangle}_{1}$ has to be defined independently of $\ell_0$.

\begin{figure}[ht]
\centering
\includegraphics[scale=0.2]{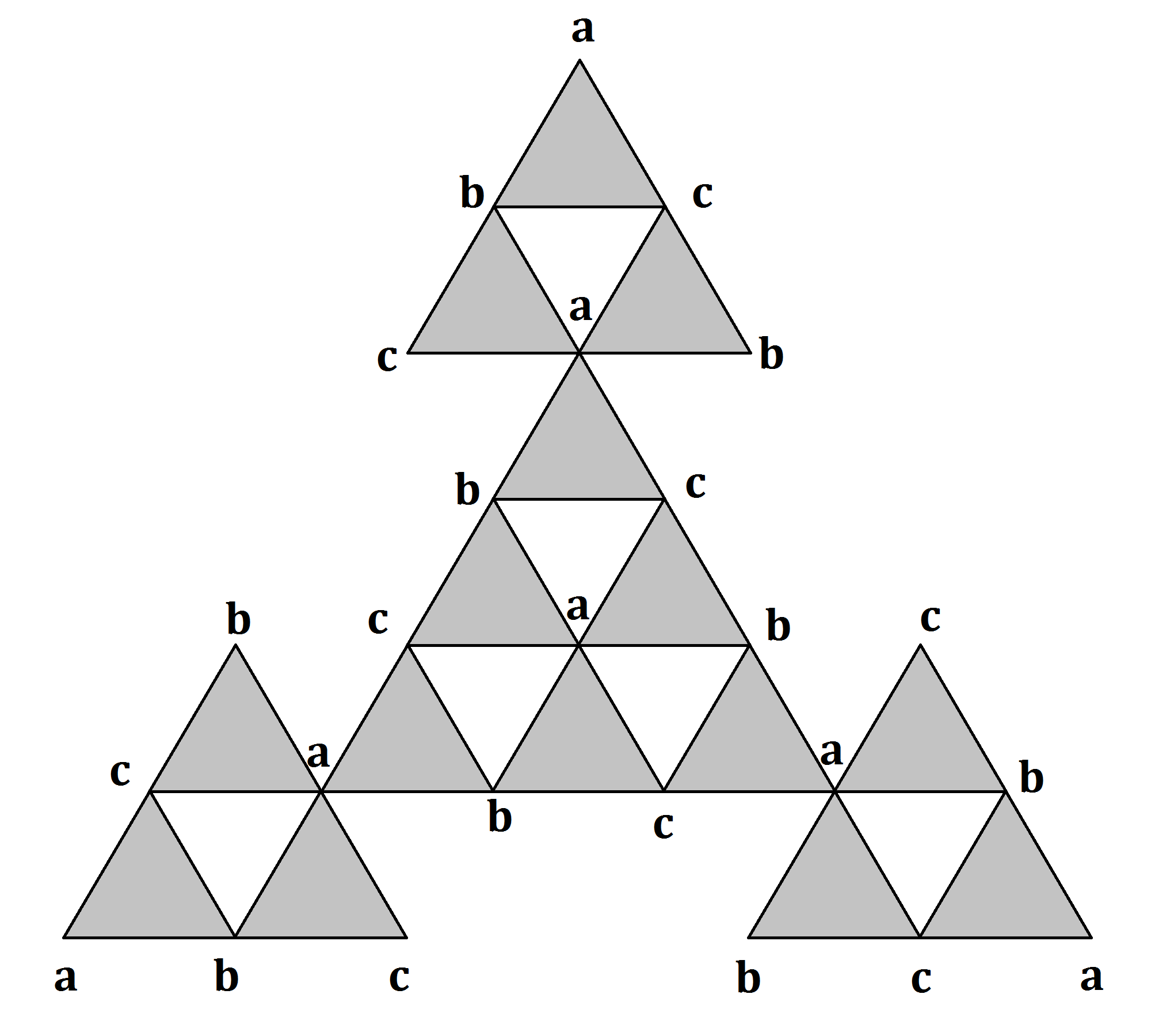}
\caption{Values of $\ell_0$ on $V^{\left\langle 1 \right\rangle}_{0}$ for the fractal with triangular complexes.}
\label{fig:triangleex}
\end{figure}
\end{example}

\begin{theorem} \label{th:all_ess}
If {$k \geq 3$} and there are no inessential fixed points of the similitudes generating $\mathcal{K}^{\langle 0\rangle},$ i.e. $k = N$,
then $\mathcal{K}^{\left\langle \infty \right\rangle}$ has the GLP.
\end{theorem}

\begin{proof}
{
Let $V^{\left\langle 0 \right\rangle}_{0} = \left\{x_1, ..., x_k\right\}$ and let $x_1, ..., x_k$ be ordered counter-clockwise (i.e. $x_j$ and $x_{j+1}$ are the endpoints of an edge of the polygon spanned by $V^{\left\langle 0 \right\rangle}_{0}$). Without losing generality, we can and will assume that $(0,0) = x_1 \in V^{\left\langle 0 \right\rangle}_{0}$. For simplicity let us also assume that $x_i$ is the fixed point of a similitude $\Psi_i$, i.e., $\Psi_i(x) = (1/L) x + \nu_i$.
}

The assumption $k = N$ implies that the complexes of a given generation form a 'ring' structure (Figures \ref{fig:ring1}, \ref{fig:ring2}). More precisely, the $1$-complex $\mathcal{K}^{\left\langle 1 \right\rangle}$ is composed of $k$ $0$-complexes $\Delta^{(1)}_0, \Delta^{(2)}_0,...,\Delta^{(k)}_0$ (arranged circularly) such that $\Delta^{(1)}_0 = \mathcal{K}^{\left\langle 0 \right\rangle}$, $\Delta^{(i)}_0 = \mathcal{K}^{\left\langle 0 \right\rangle} + L\nu_i$, $i=2,...,k$,
  and $\mathcal{K}^{\left\langle 1 \right\rangle} \cap \Delta^{(i)}_0 = \left\{x_i+L\nu_i\right\}$. In particular, any of the two edges of the complex $\Delta^{(i)}_0$ that meet at $x_i+L\nu_i$, are parallel to the corresponding edge of
$\mathcal{K}^{\left\langle 1 \right\rangle}$. This also shows that $N(=k)$ is not divisible by 4. Indeed, otherwise, the polygons spanned e.g. by the sets of the vertices $V\left(\Delta^{(1)}_0\right)$ and $V\left(\Delta^{(2)}_0\right)$, respectively, would necessary have a common edge, perpendicular to one of the edges of the polygon spanned by $V^{\left\langle 0 \right\rangle}_{0}$ (see Figure \ref{fig:oct1}). This would clearly contradict the nesting property.

\begin{figure}[ht]
\centering
\includegraphics[scale=0.4]{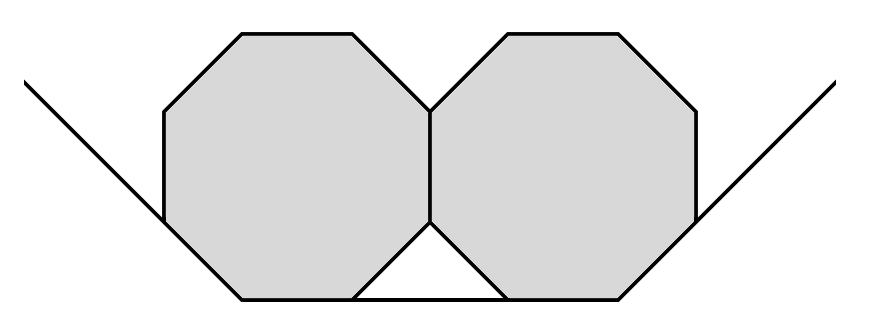}
\caption{If $N$ is divisible by 4 (e.g. $N=8$), then neighbor complexes necessarily share s vertical edge.}
\label{fig:oct1}
\end{figure}

\begin{figure}[ht]
\centering
\includegraphics[scale=0.25]{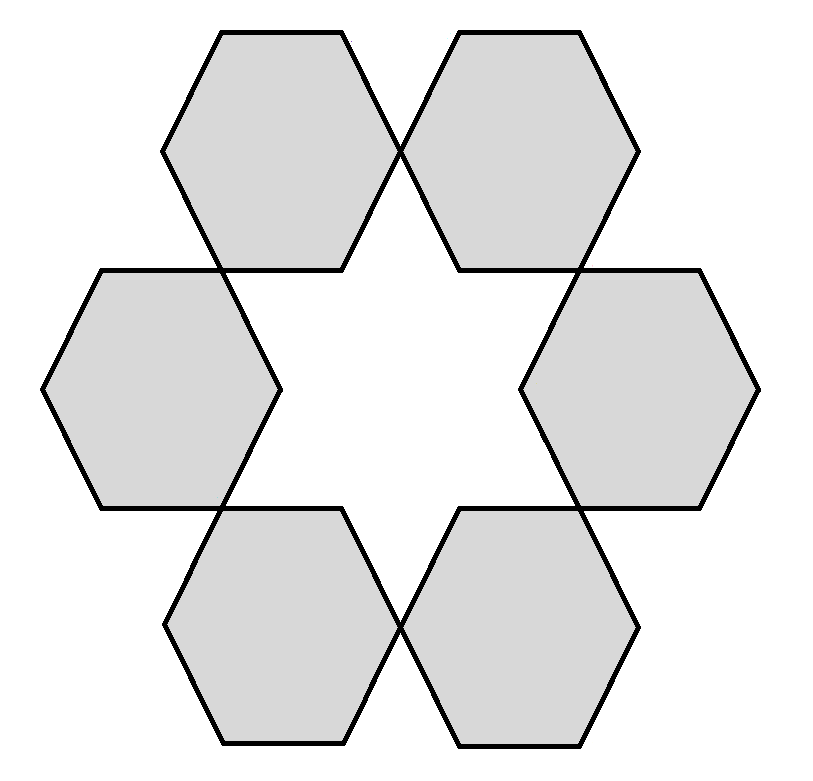}
\caption{First iteration of construction in case of six essential fixed points.}
\label{fig:ring1}
\end{figure}

\begin{figure}[ht]
\centering
\includegraphics[scale=0.15]{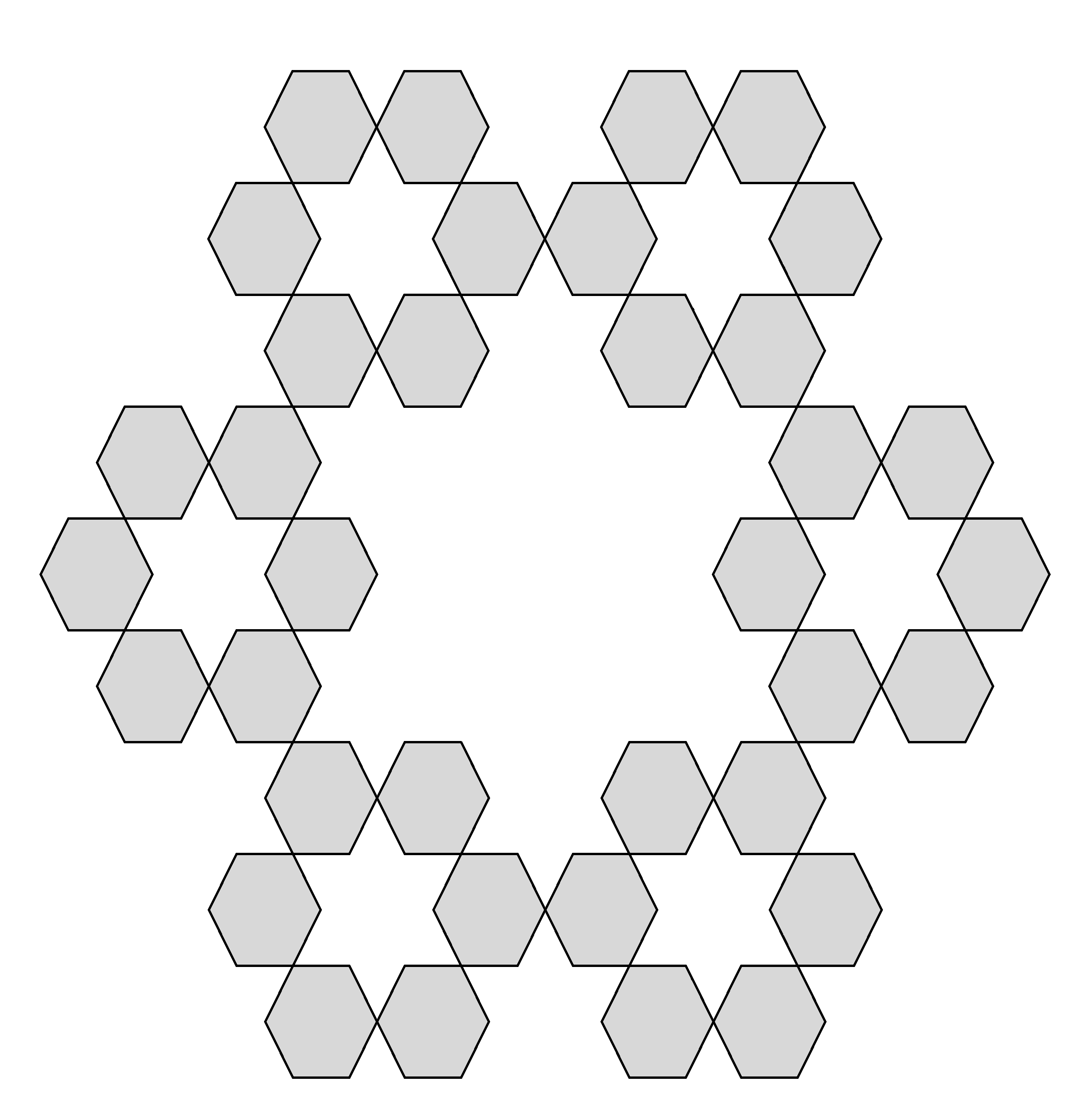}
\caption{Second iteration of construction in case of six essential fixed points.}
\label{fig:ring2}
\end{figure}




We are now in a position to establish the GLP.
First, we label the set  $V^{\left\langle 0 \right\rangle}_{0}$ as follows:
\begin{equation*}
\ell_0 \left( x_i \right) = a_i, \quad x_i \in V^{\left\langle 0 \right\rangle}_{0}.
\end{equation*}
Due to Proposition \ref{pro:glp} it suffices to show that the function $\ell_0$ extends to a function $\widetilde \ell_0: V^{\left\langle 1 \right\rangle}_{0} \to \cA$ such that the condition \eqref{eq:aux_agr} holds (we assume that $M=0$).

The $0$-complex $\Delta^{(1)}_0=\mathcal{K}^{\left\langle 0 \right\rangle}$ meets its counter-clockwise neighbor $0$-complex $\Delta^{(2)}_0$ at $x_r$, and its clockwise neigbour
$\Delta^{(k)}_0$ at $x_{k-r+2}$ (this is a consequence of the symmetry axiom of  nested fractals).The complex $\Delta^{(1)}_0$ is already labeled.
Having labeled the complex $\Delta^{(l)}_0,$ for some $l=1,...,k-1$
we can extend this labeling to $\Delta^{(l+1)}_0,$ starting with the vertex common with  $\Delta^{(l)}_0,$ and going cyclically counter-clockwise along its vertices: $a_1\to a_2\to\cdots\to a_{k}\to a_{1}.$ Proceeding this way $l-1$ times, we will label all the $0-$vertices inside $\mathcal K^{\langle 1\rangle}.$  Observe that the vertex common to $\Delta^{(l)}_0$ and $\Delta^{(l+1)}_0$ will be assigned the label $a_{(r+ (l-1)\cdot 2(r-1))({{\rm mod} \, k})}.$ All the vertices will be labeled once, except for the vertex $x_{k-r+2}$ common to $\Delta^{(1)}_0$ and $\Delta^{(k)}_0.$
The new label on this vertex will be $a_{(r+ (k-1)\cdot 2(r-1))({{\rm mod} \, k})}.$
The old label was $a_{k-r+2}.$ As $r+(k-1) \cdot 2(r-1)\equiv (k-r+2) ({\rm mod} \, k)$, the two labels agree and so we have constructed a proper extension $\widetilde \ell_0$ of $\ell_0$ from $V^{\left\langle 0 \right\rangle}_{0}$ to $V^{\left\langle 1 \right\rangle}_{0}$.

\end{proof}

Our theorems above give sufficient, but not necessary conditions for the GLP. Below we present a theorem which characterizes the fractals with the GLP among those with an even number of essential fixed points.

\begin{theorem} \label{th:even_ess}
If $\# V_{0}^{\left\langle 0\right\rangle}=k$, $k \geq 3$, is an even number, then $\mathcal{K}^{\left\langle \infty \right\rangle}$ has GLP if and only if the $0$-complexes inside the $1$-complex $\mathcal{K}^{\left\langle 1 \right\rangle}$ can be split into two
disjoint classes such that each complex from one of the classes intersects only complexes from the other class.
\end{theorem}

\begin{proof}
Let $k>2$ be an even number and let us assume that the $0$-complexes inside the $1$-complex $\mathcal{K}^{\left\langle 1 \right\rangle}$ can be split in two classes $\mathcal{T}_{0}'$ and $\mathcal{T}_{0}''$ such that each complex from one of those classes intersects only  complexes from the other class.

Without losing generality we can assume that $\mathcal{K}^{\left\langle 0 \right\rangle} \in \mathcal{T}_{0}^{'}$ and that its vertices are labelled counter-clockwise $a_1, ..., a_k$. Denote this labelling by $\ell_0.$ We reproduce these labels on $0$-complexes $\Delta_0' = \mathcal{K}^{\left\langle 0 \right\rangle} + L \nu_i \in \mathcal{T}_{0}'$ by
\begin{equation*}
\ell_0 \left(x \right) = \ell_0 \left(x-L \nu_i \right), \quad x \in V\left( \mathcal{K}^{\left\langle 0 \right\rangle} + L \nu_i \right),
\end{equation*}
i.e. the corresponding rotation $R_{\Delta_0'}$ is just the identity.

Take a $0$-complex $\Delta_0'' = \mathcal{K}^{\left\langle 0 \right\rangle} + L \nu_j \in \mathcal{T}_{0}''.$  It can be obtained by a  rotation of some $0$-complex $\Delta_0'  \in \mathcal{T}_{0}'$ \emph{by the angle $\pi$ around their intersection point}, i.e. the corresponding rotation $R_{\Delta_0''}$ is $R_{\frac{k}{2}}$, the rotation by the angle $\frac{2 \pi \frac{k}{2}}{k}= \pi$ around the barycenter of $\mathcal{K}^{\left\langle 0 \right\rangle}.$

For $ x \in V\left( \mathcal{K}^{\left\langle 0 \right\rangle} + L \nu_j \right)$ we put
\begin{equation*}
\ell_0 \left(x \right) = \ell_0 \left( R_{\frac{k}{2}} \left(x-L \nu_j\right) \right).
\end{equation*}

 This definition assigns a unique label to each  vertex. Indeed,  if $x \in \Delta_0' \cap \Delta_0''$, where $\Delta_0' = \mathcal{K}^{\left\langle 0 \right\rangle} + L \nu_i$, $\Delta_0'' = \mathcal{K}^{\left\langle 0 \right\rangle} + L \nu_j$, then $x-L\nu_i$ and $x-L \nu_j$ are symmetric images of each other  in the point reflection with respect to the barycenter of $\mathcal{K}^{\left\langle 0 \right\rangle}$ so that
\begin{equation*}
\ell_0 \left(x-L \nu_i\right) =  \ell_0 \left( R_{\frac{k}{2}} \left(x-L \nu_j\right) \right)
\end{equation*}

We thus have extended $\ell_0$ from $V^{\left\langle 0 \right\rangle}_{0} = V(\mathcal{K}^{\left\langle 0 \right\rangle})$ to $V^{\left\langle 1 \right\rangle}_{0}$ in a proper way so that the assumptions of Proposition \ref{pro:glp} are satisfied. This gives the GLP  and completes the proof of the first part.

To get the opposite implication, assume that $\mathcal{K}^{\left\langle \infty \right\rangle}$ has the GLP provided by the labelling function $\ell_0$.
As $k$ is even, each $0$-complex is an image of a neighboring $0$-complex in the rotation around their intersection point  by the angle $\pi$ . This means that for each $0$-complex $\Delta_0$ the corresponding rotation $R_{\Delta_0}$ is the identity or it is equal to $R_{\frac{k}{2}}$, the rotation by the angle $\pi$ around the barycenter of $\mathcal{K}^{\left\langle 0 \right\rangle}$.

Set
\begin{equation}
\mathcal{T}_{0}{'} = \left\{ \Delta_0 \in \mathcal{T}_0 : \Delta_0 \subset \mathcal{K}^{\left\langle 1 \right\rangle}, R_{\Delta_0} = \Id \right\},
\end{equation}
\begin{equation}
\mathcal{T}_{0}{''} = \left\{ \Delta_0 \in \mathcal{T}_0 : \Delta_0 \subset \mathcal{K}^{\left\langle 1 \right\rangle}, R_{\Delta_0} = R_{\frac{k}{2}} \right\}.
\end{equation}

No two intersecting $0$-complexes can be included in the same class $\mathcal{T}_{0}{'}$ or $\mathcal{T}_{0}{''}$, because in such a situation their common vertex would have two different labels, what is not possible. Therefore the classes $\mathcal{T}_{0}'$ and $\mathcal{T}_{0}{''}$ have the desired property.
\end{proof}

\begin{example}
On the Figure \ref{fig:hexagonex} we can see 
 $0$-complexes forming the $1$-complex $\mathcal{K}^{\left\langle 1 \right\rangle}$ for the fractal with $k=6$, $N=42$, $L=9$. The $0$-complexes are labelled according to the examples on the right-hand side of the figure.

\begin{figure}[ht]
\centering
\includegraphics[scale=0.2]{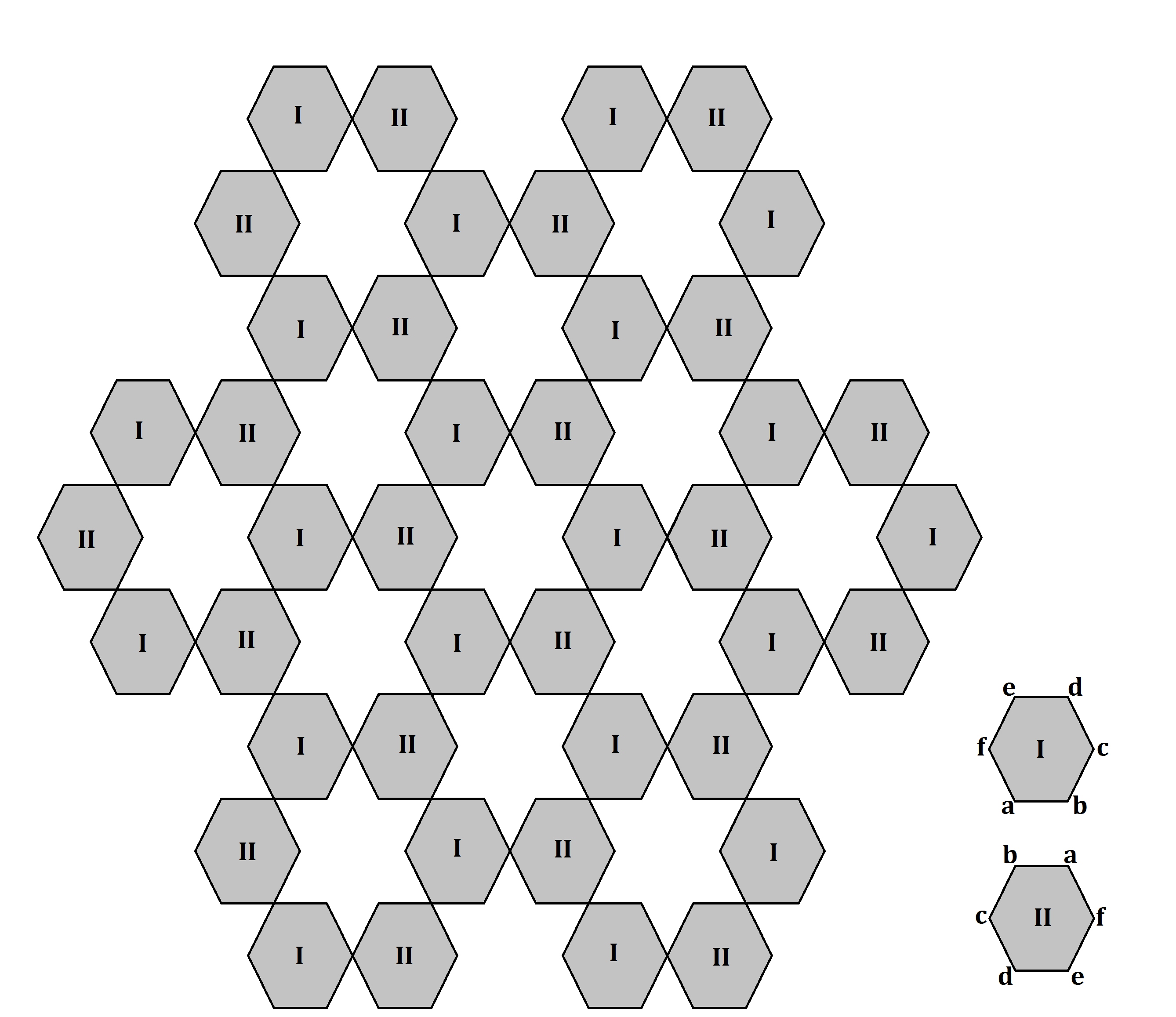}
\caption{Two classes of $0$-complexes of the fractal with hexagonal complexes and their labelling.}
\label{fig:hexagonex}
\end{figure}
\end{example}

\begin{corollary}\label{coro:squares}
If $k = 4$, then $\mathcal{K}^{\left\langle \infty \right\rangle}$ has the GLP.
\end{corollary}

\begin{proof}
If there are four essential fixed points, then the vertices of any complex form a square. Without loss of generality we can assume that
\begin{equation}
V^{\left\langle 0\right\rangle}_{0} = \left\{ \left(0,0\right), \left(1,0\right), \left(1,1\right),\left(0,1\right)\right\}
\end{equation}
and set
\begin{equation*}
\ell_0\left(\left(0,0\right)\right) = a, \quad \ell_0\left(\left(1,0\right)\right) = b, \quad \ell_0\left(\left(1,1\right)\right) = c, \quad \ell_0\left(\left(0,1\right)\right) = d.
\end{equation*}

Observe that $V^{\left\langle \infty \right\rangle}_{0} \subset \mathbb{Z}e_1 + \mathbb{Z}e_2$, where $e_1=\left(0,1\right)$ and $e_2 = \left(1,0\right)$.

Let $\Delta_0= \mathcal{K}^{\left\langle 0 \right\rangle} + \sum_{j=1}^{J} L^{j} \nu_{i_j} $ be a $0$-complex. Then it can be also uniquely represented as
\begin{equation*}
\Delta_0  =  \mathcal{K}^{\left\langle 0 \right\rangle} + n_1 e_1 + n_2 e_2
\end{equation*}
for some $n_1, n_2 \in \mathbb{Z}$.

Due to the nesting property, $n_1$ and $n_2$ are either both odd or both even, as otherwise the neighboring complexes would share a common side, not only the vertices. It allows us to use the representation
\begin{equation}
\label{eq:squareaddress}
\Delta_0 =  \mathcal{K}^{\left\langle 0 \right\rangle} + \frac{n_1+n_2}{2} \left(e_1 + e_2\right) + \frac{n_1-n_2}{2} \left(e_1 - e_2\right),
\end{equation}
and then set the two classes of $0$-complexes as follows:
\begin{equation}
\mathcal{T}_{0}{'} = \left\{  \Delta_0 \in \mathcal{T}_0 : \Delta_0 = \mathcal{K}^{\left\langle 0 \right\rangle} + \frac{n_1+n_2}{2} \left(e_1 + e_2\right) + \frac{n_1-n_2}{2} \left(e_1 - e_2\right) \in \mathcal{K}^{\left\langle 1 \right\rangle} ,  \frac{n_1+n_2}{2} \in 2\mathbb{Z} \right\},
\end{equation}
\begin{equation}
\mathcal{T}_{0}{''} = \left\{  \Delta_0 \in \mathcal{T}_0 : \Delta_0 = \mathcal{K}^{\left\langle 0 \right\rangle} + \frac{n_1+n_2}{2} \left(e_1 + e_2\right) + \frac{n_1-n_2}{2} \left(e_1 - e_2\right) \in \mathcal{K}^{\left\langle 1 \right\rangle} ,  \frac{n_1+n_2}{2} \in 2\mathbb{Z}+1 \right\}.
\end{equation}

Clearly a complex with odd coefficients in the  representation \ref{eq:squareaddress} can intersect only these with even coefficients and vice versa.
\end{proof}

\begin{example}
On Figure \ref{fig:squareex} we see the $0$-complexes forming the $1$-complex $\mathcal{K}^{\left\langle 1 \right\rangle}$ for the Vicsek cross ($k=4$, $N=5$, $L=3$). With $\ell_0$ given on $V^{\left\langle 0 \right\rangle}_{0}$ (labelling of the vertices of bottom left square), we can label all vertices from $V^{\left\langle 1 \right\rangle}_{0}$ in the way preserving the orientation of labels on each $0$-complex.


\begin{figure}[ht]
\centering
\includegraphics[scale=0.4]{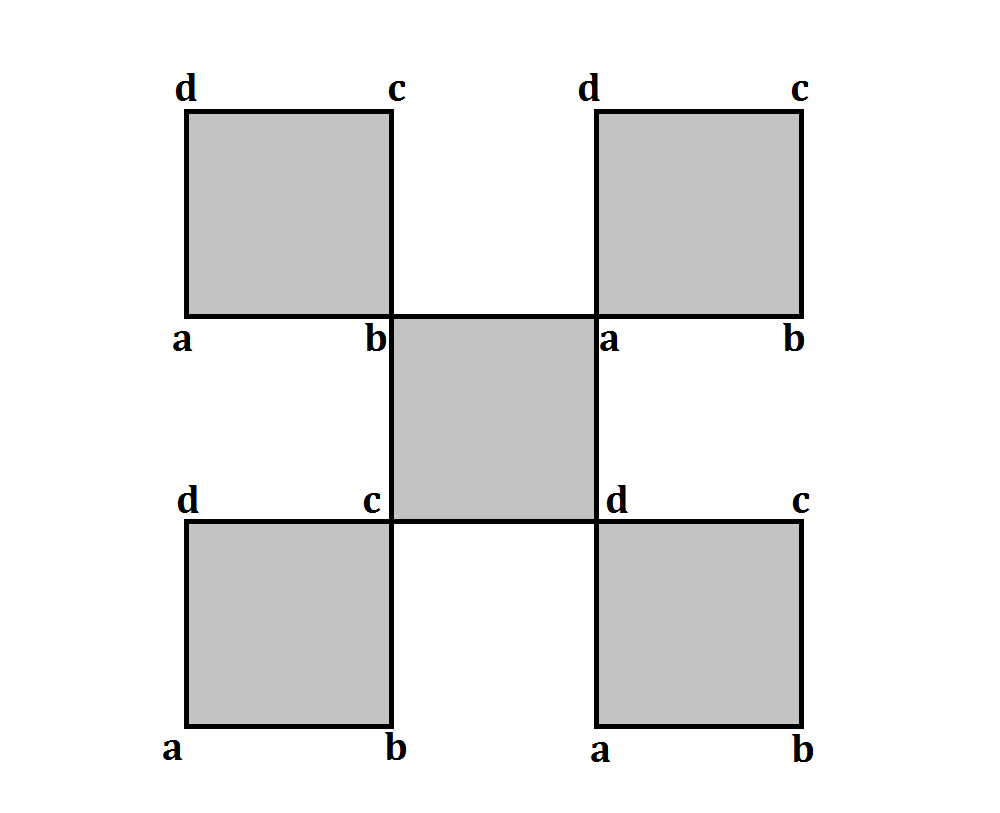}
\caption{Values of $\ell_0$ on $V^{\left\langle 1 \right\rangle}_{0}$ for the Vicsek cross.}
\label{fig:squareex}
\end{figure}
\end{example}

Summarizing, the fractals not considered in any of the theorems above are the fractals for which $k$ is odd 
and $N>k$. Below we present an example of such a fractal that cannot be well-labelled.

\begin{example}
 Figure \ref{fig:nonagon2} presents the shape of $\mathcal{K}^{\left\langle 1 \right\rangle}$ of a fractal with nonagonal complexes ($k=9$) and $N=54$.

\begin{figure}[ht]
\centering
\includegraphics[scale=0.1]{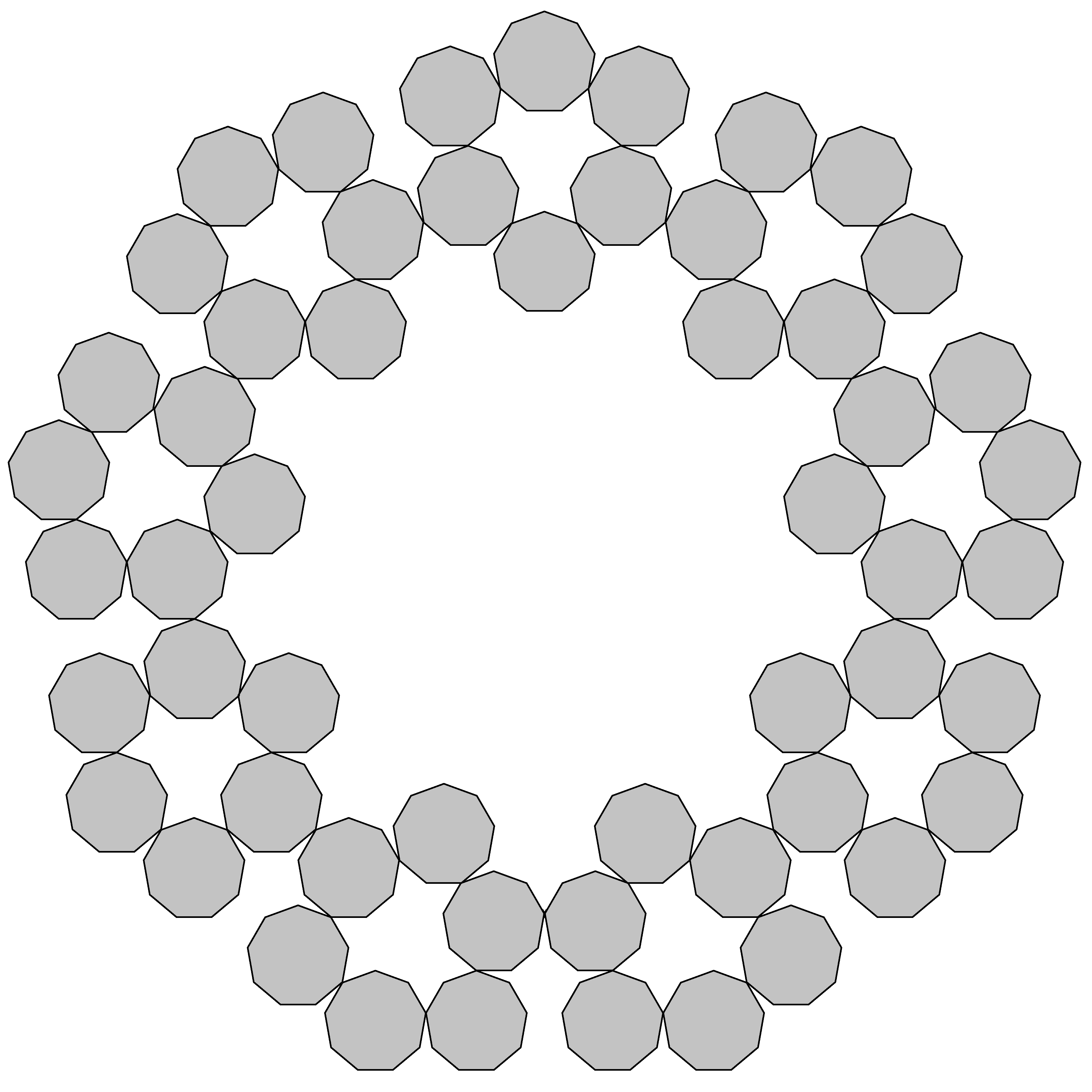}
\caption{Connected vertices from $V^{\left\langle 1 \right\rangle}_{0}$ for the fractal with nonagonal complexes.}
\label{fig:nonagon2}
\end{figure}

 Figure \ref{fig:nonagon1} is a close up of the part near one vertex, which itself cannot be well labelled.

\begin{figure}[ht]
\centering
\includegraphics[scale=0.2]{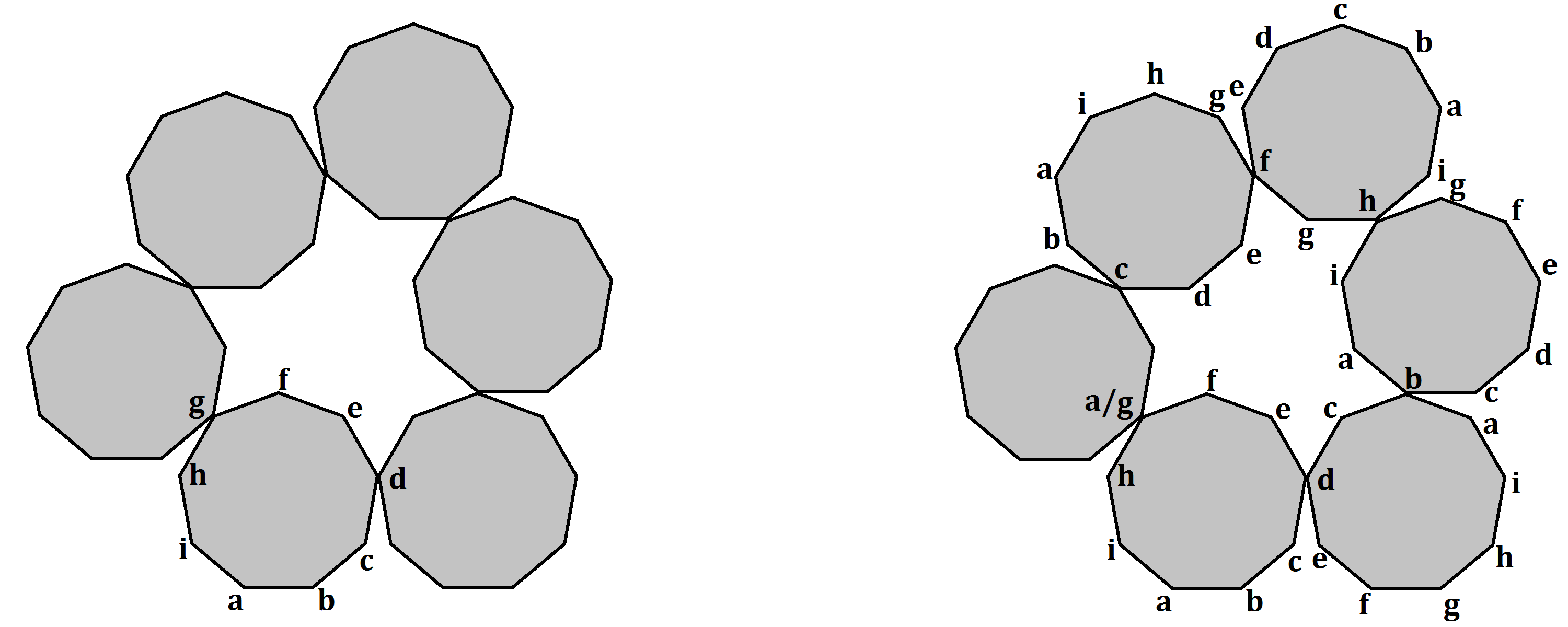}
\caption{The attempt to label vertices of six $0$-complexes of the fractal above.}
\label{fig:nonagon1}
\end{figure}


We label the vertices of the bottom leftmost complex counter-clockwise using labels $a, b, c, ..., i$. Then we put labels on adjacent complexes according to a proper rotation.

If there existed a good labelling of this fractal, then (because of its uniqueness) we would have obtained it by such labelling. But let us see that the last unlabelled complex has a vertex with label $c$ in the intersection with its top neighbor. This means that the vertex in its intersection with another complex should be labelled as $a$, while it is already labelled as $g$. The vertex cannot have two labels, therefore the good labelling of such fractal is impossible.

\end{example}

\section{Reflected Brownian motion on USNFs}\label{sec:reflected}

\subsection{The process on $\mathcal K^{\langle\infty\rangle}$.}

Let $Z=(Z_t, \mathbf{P}^{x})_{t \geq 0, \, x \in \mathcal{K}^{\left\langle \infty \right\rangle}}$ be \emph{the Brownian motion} on the USNF $\mathcal{K}^{\left\langle \infty \right\rangle}$ \cite{bib:Lin, bib:Kus2}. Such a process has been constructed by means of Dirichlet forms \cite{bib:Fuk1, bib:Kum}. It is a strong Markov, Feller process with continuous trajectories, whose distributions are invariant under local isometries of $\mathcal{K}^{\left\langle \infty \right\rangle}$. It has transition probability densities $g(t,x,y)$ with respect to the $d_f$-dimensional Hausdorff measure $\mu$ on $\mathcal{K}^{\left\langle \infty \right\rangle}$ (recall that $\mu(\mathcal{K}^{\left\langle 0 \right\rangle}) = 1$). More precisely, one has
$$
\mathbf{P}^{x}(Z_t \in A) = \int_A g(t,x,y) \mu(dy), \quad t > 0, \ \ x \in \mathcal{K}^{\left\langle \infty \right\rangle}, \ \ A \subset \cB(\mathcal{K}^{\left\langle \infty \right\rangle}).
$$
Densities $g(t,x,y)$ are jointly continuous on $(0,\infty) \times \mathcal{K}^{\left\langle \infty \right\rangle} \times \mathcal{K}^{\left\langle \infty \right\rangle}$ and satisfy the scaling property
$$
g(t,x,y) = L^{d_f} g(L^{d_w} t, L x, L y), \quad t>0, \ \ x, y \in \mathcal{K}^{\left\langle \infty \right\rangle}.
$$
Moreover, they enjoy the following \emph{sub-Gaussian estimates}: there are absolute constants $C_1,...,C_4 >0$ such that \cite[Theorems 5.2, 5.5]{bib:Kum}
\begin{multline}
\label{eq:kum}
C_{1} t^{-d_s/2} \exp \left(-C_{2} \left(\frac{\left|x-y \right|^{d_w}}{t} \right)^{\frac{1}{d_J -1}} \right) \leq g(t,x,y) \\
\leq C_{3} t^{-d_s/2} \exp \left(-C_{4} \left(\frac{\left|x-y \right|^{d_w}}{t} \right)^{\frac{1}{d_J -1}} \right), \quad t>0, \ \ x,y \in \mathcal{K}^{\left\langle \infty \right\rangle}.
\end{multline}
Recall that $d_w$ and $d_s=2d_f/d_w$ are the walk and the spectral dimensions of $\mathcal{K}^{\left\langle \infty \right\rangle}$, respectively, and $d_J > 1$ is a so-called \emph{chemical exponent} of $\mathcal{K}^{\left\langle \infty \right\rangle}$. The regularity properties of the densities $g$ and the bounds \eqref{eq:kum} has been established by T. Kumagai for general nested fractals under the following assumption. Let $d_M$ denote the graph metric of order $M$ on $\mathcal{K}^{\left\langle \infty \right\rangle}$.
Then there exists $n \in \mathbb{N}$ such that for any $M\in\mathbb Z,$  if $x,y \in \mathcal{K}^{\left\langle \infty \right\rangle}$ satisfy $\left|x-y\right| \leq L ^{M}$, then $d_M(x,y) \leq n$ (\cite[Sec. 5]{bib:Kum}).
In our setting this assumption is always satisfied. See comments following Lemma \ref{lem:metrics} in the Appendix.
The constant $d_J$ has been introduced in the cited paper as a parameter describing the \emph{shortest path scaling} on a given nested fractal 
Typically $d_w \neq d_J$, but it is known that in the case of Sierpi\'nski gasket one has $d_w = d_J$. Note that very often, due to very rich geometric structure of $\mathcal{K}^{\left\langle \infty \right\rangle}$, the shortest path (or geodesic) metric cannot be well defined.

Diffusion processes on fractals with transition densities having sub-Gaussian estimates are often called \emph{fractional diffusions} \cite{bib:Bar}. 

\subsection{Construction of the reflected Brownian motion}

Suppose now that the unbounded fractal $\mathcal K^{\langle\infty\rangle}$ has the GLP. For an arbitrary $M \in \Z$ we will define the \emph{reflected Brownian motion} on an $M$-complex $\mathcal{K}^{\left\langle M\right\rangle}.$  Indeed, as it will be seen below the existence of such a process is a consequence of the well-definiteness of the projection operation $\pi_M: \mathcal{K}^{\left\langle \infty  \right\rangle} \to \mathcal{K}^{\left\langle M\right\rangle}$ introduced in Section 3.  

We will first construct a regular enough version of the transition probability densities for the process in question.
Our construction is a generalization of that in \cite{bib:KPP-PTRF}, which was performed for the unit complex of the planar Sierpi\'nski triangle (see also \cite{bib:KaPP2}). We would like to emphasize that our present case of general USNFs with GLP is much more delicate and it requires substantial modifications of the previous argument.

Throughout this section we assume that $M \in \Z$ is arbitrary, but fixed. The reflected Brownian motion on $\mathcal{K}^{\left\langle M\right\rangle}$ is defined canonically by
\begin{equation}\label{eq:process-def}
Z_t^M = \pi_M(Z_t),
\end{equation}
where $\pi_M: \mathcal{K}^{\left\langle \infty  \right\rangle} \to \mathcal{K}^{\left\langle M\right\rangle}$ is the projection from in Section 3. Formally, we will investigate the stochastic process $(Z_t^M, \mathbf{P}^{x}_{M})_{t \geq 0, \, x \in \mathcal{K}^{\left\langle M\right\rangle}}$, where the measures $\mathbf{P}^{x}_{M}$, $x \in \mathcal{K}^{\left\langle M\right\rangle}$, are defined as the projections of the measures $\mathbf{P}^{x}$, $x \in \mathcal{K}^{\left\langle M\right\rangle}$, determining the distribution of the free Brownian motion. The finite dimensional distributions
of $Z^M$ are given by
\begin{align} \label{eq:fdd_reflected}
\mathbf{P}^{x}_{M}(Z^M_{t_1} \in A_1, ... , Z^M_{t_n} \in A_n)= \mathbf{P}^{x}(Z_{t_1} \in \pi_M^{-1}(A_1), ... , Z_{t_n} \in \pi_M^{-1}(A_n)),
\end{align}
for every $0 \leq t_1 < t_2 <...< t_n$, $x \in \mathcal{K}^{\left\langle M\right\rangle}$ and $A_1,...,A_n \in \cB(\mathcal{K}^{\left\langle M\right\rangle})$. Note that in fact the projections of the measures $\mathbf{P}^{x}$ (denoted by $\pi_M(\mathbf{P}^{x})$) are well defined for every $x \in \mathcal{K}^{\left\langle \infty \right\rangle}$ and the right hand side of \eqref{eq:fdd_reflected} defines the finite dimensional distributions for such measures in general case.

From the definition of the measures $\mathbf{P}^{x}_{M}$ it is obvious that the one-dimensional distributions of the process $Z^M$ are absolutely continuous with respect to the Hausdorff measure $\mu$ restricted to the complex $\mathcal{K}^{\left\langle M\right\rangle}$.

Define the function $g_M(t,x,y): (0,\infty) \times \mathcal{K}^{\left\langle \infty \right\rangle} \times \mathcal{K}^{\left\langle M\right\rangle} \to (0,\infty)$ by
\begin{equation}
\label{eq:refldens}
g_{M}\left(t,x,y\right)= \left\{ \begin{array}{ll}
\displaystyle\sum_{y'\in \pi_{M}^{-1} (y)}{ g(t,x,y')} & \textrm{if } y \in \mathcal{K}^{\left\langle M\right\rangle} \backslash V_{M}^{\left\langle M\right\rangle}, \\
\displaystyle\sum_{y'\in \pi_{M}^{-1} (y)}{g(t,x,y')} \cdot \textrm{rank}(y') & \textrm{if } y \in V_{M}^{\left\langle M\right\rangle}, \\
\end{array}\right.
\end{equation}
where $\textrm{rank}(y')$ is the number of $M$-complexes meeting at the point $y'$. We see from \eqref{eq:fdd_reflected} that the functions $g_M(t,x,\cdot)$, $x \in \mathcal{K}^{\left\langle M\right\rangle}$, are indeed versions of the densities of the measures
$$
P_M(t,x,A) = \mathbf{P}^{x}(Z_t \in \pi_M^{-1}(A)), \quad t >0, \ \ x \in \mathcal{K}^{\left\langle M\right\rangle}, \ \ A \in \cB(\mathcal{K}^{\left\langle M\right\rangle}),
$$
which are natural candidates for the transition probabilities of the process $Z^M$ (observe that  $g_M(t,x,\cdot)$ are versions of densities for the projected measures $\mathbf{P}^{x}(Z_t \in \pi_M^{-1}(\cdot))$ for every $x \in \mathcal{K}^{\left\langle \infty\right\rangle}$). We will prove below that this choice of $g_M$ will provide us with further regularity properties of $Z^M$ like Markov, Feller and strong Feller properties. We also would like to note that the definition and the regularity of $g_M$ strongly depend on the geometric properties of a given USNF. If $k=3$, then for any vertex $y',$  $\textrm{rank}(y') \in \left\{1,2,3\right\}$, and if $k\geq 4$, then $\textrm{rank}(y') \in \left\{1,2\right\}$, and they can vary from point to point. For the unbounded one-sided Sierpi\'nski triangle, every vertex outside the origin has rank $2$, and so the situation is much simpler than the general one.

We are now in a  position to state our main result in this section.  For $t>0$ and $f \in L^{\infty}(\mathcal{K}^{\left\langle M\right\rangle})$ let $$T_t^M f(x) = \int_{\mathcal{K}^{\left\langle M\right\rangle}} g_M(t,x,y) f(y)  \mu(dy),  \quad x \in \mathcal{K}^{\left\langle M\right\rangle}.$$

\begin{theorem} \label{thm:main1}
Let $\mathcal K^{\langle\infty\rangle}$ be an USNF with the GLP.
Let $M\in\mathbb Z.$ For the functions $g_M$ defined in \eqref{eq:refldens}
the following hold.
\begin{itemize}
\item[(1)] The function $g_M(t,x,y)$ is continuous on $(0,\infty) \times \mathcal{K}^{\left\langle M\right\rangle} \times \mathcal{K}^{\left\langle M\right\rangle}$ and bounded on {$[u,v] \times \mathcal{K}^{\left\langle M\right\rangle} \times \mathcal{K}^{\left\langle M\right\rangle}$, for every $0<u < v <\infty$}. In particular, $T_t^M \big(L^{\infty}(\mathcal{K}^{\left\langle M\right\rangle}) \subset C_b(\mathcal{K}^{\left\langle M\right\rangle})$, for every $t>0$.
\item[(2)] For every $f \in C_b(\mathcal{K}^{\left\langle M\right\rangle})$ we have $\left\|T^M_t f - f \right\|_{\infty} \to 0$ as $t \to 0^{+}$.
\item[(3)] For every $t, s >0$  and $x, y \in \mathcal{K}^{\left\langle M\right\rangle}$ we have
$$
g_M(t+s,x,y) = \int_{\mathcal{K}^{\left\langle M\right\rangle}} g_M(t,x,z) g_M(s,z,y) \mu(dz).
$$
\item[(4)] For every $t >0$  and $x, y \in \mathcal{K}^{\left\langle M\right\rangle}$ we have
$$
g_M(t,x,y) = g_M(t,y,x).
$$
\end{itemize}
\end{theorem}

The next theorem is a direct consequence of the above result.

\begin{theorem} \label{thm:main2}
The process $(Z_t^M, \mathbf{P}^{x}_{M})_{t \geq 0, \, x \in \mathcal{K}^{\left\langle M\right\rangle}}$ defined by \eqref{eq:process-def} is a continuous Markov process with transition probability densities $g_M(t,x,y)$, which is Feller and strong Feller.
\end{theorem}
We postpone the proof of Theorem \ref{thm:main1} till the end of this section.

\begin{proof}[Proof of Theorem \ref{thm:main2}]
First note that by Theorem \ref{thm:main1} (3), we immediately derive from the general theory of Markov processes that the process in question is a continuous Markov process on $\mathcal{K}^{\left\langle M\right\rangle}$ with transition probabilities given by
$$
P_M(t,x,A) = \int_{\mathcal{K}^{\left\langle M\right\rangle}} g_M(t,x,y) \mu(dy), \quad t >0, \ \ x \in \mathcal{K}^{\left\langle M\right\rangle}, \ \ A \in \cB(\mathcal{K}^{\left\langle M\right\rangle}),
$$
where $g_M$ are given by \eqref{eq:refldens}. Theorem \ref{thm:main1} (1)-(2) also gives that $Z_M$ is Feller and strong Feller process.
\end{proof}

The proof of Theorem \ref{thm:main1} will be given at the end of this section after a sequence of auxiliary results, which we prove below.

\begin{lemma}
\label{lem:properties} We have the following.
\begin{enumerate}
\item For every $0<u<v<\infty$, the series $\sum_{y'\in \pi_{M}^{-1} (y)}{g(t,x,y')}$ is uniformly convergent in $(t,x,y)$ on $[u,v] \times \mathcal{K}^{\left\langle \infty\right\rangle} \times \mathcal{K}^{\left\langle M\right\rangle}$.
\item The function $g_{M}(t,x,y)$ defined by \eqref{eq:refldens} is continuous on {$(0,\infty) \times \mathcal{K}^{\left\langle \infty\right\rangle} \times \mathcal{K}^{\left\langle M\right\rangle}$}.
\end{enumerate}
\end{lemma}

\begin{proof} To prove the assertion (1), we use the $M$-graph distance introduced at the end Section \ref{sec:usnf}. We may write
\begin{equation} \label{eq:fin_us}
\sum_{y'\in \pi_{M}^{-1} (y)}{g(t,x,y')} = \sum_{n=1}^{\infty} \sum_{\Delta_M\left(y'\right) \in \mathcal{L}_{M,n,x}} {g(t,x,y')} =: \sum_{n=1}^{\infty} a_{n,t,x}.
\end{equation}
Moreover, by using the upper bound in \eqref{eq:kum} (together with the identity $d_w / d_f = 2/d_s$), the distance comparison principle in Lemma \ref{lem:metrics} and the estimate in Lemma \ref{lem:complab}, we have {for $n \geq 3$}
\begin{align*}
a_{n,t,x} & \leq \# \mathcal{L}_{M,n,x} \cdot \sup_{\Delta_M\left(y'\right) \in \mathcal{L}_{M,n,x}} g\left(t,x,y'\right) \\
& \leq c_1 \ \# \mathcal{L}_{M,n,x} \ \sup_{\Delta_M\left(y'\right) \in \mathcal{L}_{M,n,x}} \left(t^{-\frac{d_f}{d_w}} \ \exp \left( -c_2 \left(\frac{d_M(x,y^{\prime})^{d_w / d_f}}{t} \right)^{\frac{1}{d_J-1}}\right)\right) \\
& \leq c_3 \ n^{d_f} \ t^{-\frac{d_f}{d_w}} \ \exp \left( -c_4 \left(\frac{n^{d_w / d_f}}{t} \right)^{\frac{1}{d_J-1}}\right).
\end{align*}
Clearly, for every $\beta, \gamma>0$ and $r_0 >0$ there exists a constant $c_5 >0$ such that $e^{-r^{\beta}} \leq c_5 r^{-\gamma}$ for $r \geq r_0$.
Since for $n \in \Z$ and {$t \in (0,v]$} the ratio $n^{d_w / d_f}/t$ is bounded away from zero, we get the estimate
\begin{align} \label{eq:for_later_use}
a_{n,t,x} \leq c_6 \ n^{d_f}t^{-\frac{d_f}{d_w}} \left( \frac{t}{n^{d_w / d_f}}\right)^{\gamma} = c_6 \, t^{\gamma - \frac{d_f}{d_w}} \, n^{-\gamma(d_w/d_f)+d_f} , \quad x\in \mathcal K^{\langle\infty\rangle},\quad  t \in (0,v], \ \ n  \geq 3.
\end{align}
Choose $\gamma$ large enough to have $\gamma(d_w/d_f) - d_f >2$ (in particular, $\gamma > d_f/d_w$). {Then
 $a_{n,t,x}\leq c_7 n^{-2}$ for every $n \geq 3$. On the other hand, we easily get from \eqref{eq:kum} that for $n=1, 2$
$$
a_{n,t,x}\leq c_8 t^{-\frac{d_f}{d_w}} \leq c_9, \quad t \in [u,v].
$$
The assertion (1) follows.}

We now prove (2). First note that if $y\notin V_{M}^{\left\langle M \right\rangle}$, then the kernel $g_M(t,x,y)$ inherits the continuity in $(0,\infty) \times \mathcal{K}^{\left\langle \infty\right\rangle} \times \mathcal{K}^{\left\langle M\right\rangle}$ from the continuity properties of the density $g$. This is a direct consequence of the uniform convergence of the series in (1).

Suppose now that $x \in \mathcal{K}^{\left\langle \infty\right\rangle}$, $y \in V_{M}^{\left\langle M \right\rangle}$, $t>0$ and that $x_n \in \mathcal{K}^{\left\langle \infty\right\rangle}$, $y_n \in \mathcal{K}^{\left\langle M \right\rangle} \backslash  V_{M}^{\left\langle M \right\rangle}$ and $t_n >0$ are such that $(x_n,y_n,t_n) \to (x,y,t)$ as $n \to \infty$. Observe that for sufficiently large $n$ and every $y_n^{\prime} \in \pi_{M}^{-1}(y_n)$ there are exactly $\text{rank}(y^{\prime})$ different points $y_{i,n}^{\prime} \in \mathcal{K}^{\left\langle \infty\right\rangle}$ (different for different $y_n^{\prime}$'s) such that $y_{i,n}^{\prime}  \to y^{\prime}$ as $n \to \infty$, for every $i=1,...,\text{rank}(y^{\prime})$. Moreover, it holds that
$$
g_{M} \left(t_n, x_n, y_n\right) = \sum_{y'\in \pi_{M}^{-1} (y)} \sum_{i=1}^{\text{rank}(y^{\prime})} g(t_n, x_n, y_{i,n}^{\prime}).
$$
Then, thanks to the uniform convergence we can pass to the limit under the sums as follows:
\begin{align*}
\lim_{\left(t_n, x_n, y_n\right) \to \left(t,x,y\right)} g_{M} \left(t_n, x_n, y_n\right) & = \sum_{y'\in \pi_{M}^{-1} (y)}  \sum_{i=1}^{\text{rank}(y^{\prime})} \lim_{\left(t_n, x_n, y_{i,n}^{\prime}\right) \to \left(t,x,y'\right)} g\left(t_n, x_n, y_{i,n}^{\prime}\right) \\
& = \sum_{y'\in \pi_{M}^{-1} (y)} \text{rank}(y^{\prime}) g\left(t, x, y'\right) = g_{M} \left(t,x,y\right).
\end{align*}
This completes the proof of the lemma.
\end{proof}

We introduce the consecutive hitting times of the $m-$th grid:
\begin{eqnarray*}
T_{m}^{(1)} &=& \inf \left\{ t>0: Z_t \in  V_{m}^{\left\langle \infty \right\rangle} \backslash \left\{ Z_0 \right\}\right\} \\
T_{m}^{(n+1)} &=& \inf \left\{ t>T^{(n)}: Z_t \in  V_{m}^{\left\langle \infty \right\rangle} \backslash \left\{Z_{T_m^{(n)}} \right\}\right\} \textrm{, for } n>1.
\end{eqnarray*}
$\left\{T_{m}^{(n)}\right\}_{n \in \mathbb{N}}$ is an increasing sequence of stopping times and $\lim_{n \to \infty} T_{m}^{(n)} = \infty$ almost surely. This is so because the number
\begin{equation}\label{eq:pointdistance}
\alpha :=\inf\{|x-y|: x,y\in V_m^{\langle\infty\rangle}\}
\end{equation}
is strictly positive.
It is also convenient to define $$T_{m}^{(0)} = \inf \left\{ t\geq0: Z_t \in  V_{m}^{\left\langle \infty \right\rangle}\right\}.$$ Clearly, for all paths starting from $x \in V_{m}^{\left\langle \infty \right\rangle}$ one has $0 = T_{m}^{(0)} < T_{m}^{(1)}$.

The following lemma is essential in our further considerations.
\begin{lemma}
\label{lem:lawseq}
Let $x,y \in \mathcal{K}^{\left\langle \infty \right\rangle}$ be such that $\pi_{M}(x)=\pi_{M}(y)$. Then the following hold.
\begin{itemize}
\item[(1)] For every $n \in \N$, $a \in \cA$ and $t>0$
\begin{equation*}
\mathbf{P}^{x}\left(T_{M}^{(n)} < t, \ell_M\left( Z_{T_{M}^{(n)}}\right)=a \right) =\mathbf{P}^{y}\left(T_{M}^{(n)} < t, \ell_M\left( Z_{T_{M}^{(n)}}\right)=a \right).
\end{equation*}
\item[(2)] For every Borel $\Gamma \in \mathcal{K}^{\left\langle M \right\rangle}$ and $t>0$
\begin{equation*}
\mathbf{P}^{x} \left(Z_t \in \pi_{M}^{-1}\left(\Gamma\right), t < T^{(1)}_M \right)=\mathbf{P}^{y} \left(Z_t \in \pi_{M}^{-1}\left(\Gamma\right), t < T^{(1)}_M \right).
\end{equation*}
\end{itemize}
\end{lemma}
\begin{proof}
We first establish (1) by using induction in $n$.

\smallskip

\noindent For $n=1$ we consider two cases.

\smallskip

\noindent \textsc{Case 1. $x,y\notin V_M^{\langle\infty\rangle}.$}
In this case the laws of $\left(T_{M}^{(1)}, \ell_M\left( Z_{T_{M}^{(1)}}\right)\right)$ depend entirely on the laws of $\left(Z_t\right)$ up to exit times from $\Delta_{M}(x), \Delta_{M}(y)$ respectively, which are identical.

\smallskip

\noindent \textsc{Case 2. $x,y \in V_{M}^{\left\langle \infty \right\rangle}.$} Let ${r_1} = \textrm{rank}(x) \in \left\{1,2,3\right\}$, ${r_2} = \textrm{rank}(y) \in \left\{1,2,3\right\}$. Let us notice that even though $\pi(x)=\pi(y),$ it is possible to have $r_1 \neq r_2$.
This feature was not present in the setting of Sierpi\'{n}ski gasket, where the rank of all the vertices was equal to 2.
To overcome this difficulty we will  reduce the problem to the analysis of the  random walk induced by the Brownian motion on  $\mathcal{K}^{\left\langle \infty \right\rangle}$.

To this end, denote by $Z^x$ the process $Z$ on $\mathcal{K}^{\left\langle \infty \right\rangle}$ starting from $x$ and consider the sequence of random walks $(Y^{m,x})_{m \in \Z}$ on $V_{m}^{\left\langle \infty \right\rangle}$, starting from $x$, given by $Y^{m,x}_k:= Z^x_{T_{m}^{k}}$with $k=0,1,2,...$ (as $x$ is already fixed, below we drop it from the notation). Such a family of random walks has a specific consistency property which is called the \emph{decimation invariance} (for more details we refer to \cite{bib:Lin,bib:Bar,bib:Kum,bib:Kr1,bib:Kr2}). Following \cite{bib:BP} and \cite[p. 208]{bib:Kum}, we infer that if we take
\begin{align} \label{eq:def_approx}
Z^m_t:= Y^m_{\left[\gamma^{-m} t\right]}, \quad t>0,
\end{align}
with an appropriate time scale parameter $\gamma$ (resulting from the construction of the process $Z$ in \cite{bib:Lin}), then $\mathbf{P}^x$-a.s.
$Z^m_t \to Z_t$ as $m \to -\infty$, uniformly on compact subsets of $[0,\infty)$ (recall that in our settings the sign of $m$ is opposite to that in the quoted papers). In particular, if $T_{M,m} = \inf \left\{ t>0: Z^m_t \in V_{M}^{\left\langle \infty \right\rangle} \backslash \{x\}\right\}$, then $\mathbf{P}^x$-a.s. $T_{M,m} \to T^{(1)}_M$ as $ m \to -\infty$ (cf. \cite[p. 208]{bib:Kum}) and, in consequence,
$$
 \mathbf{P}^{x} \left [T_{M,m} \leq t, \ell_M \left( Z^{m}_{T_{M,m}} \right) = a \right] \to \mathbf{P}^{x} \left [T^{(1)}_M \leq t, \ell_M \left( Z_{T^{(1)}_M} \right) = a \right],
$$
for any given $a \in \cA$ and $t>0$. Exactly the same argument leads to the analogical convergence under the measure $\mathbf{P}^{y}$. Denote by $\tau_M^m$ the consecutive hitting times of the $M-$th grid by the random walk $Y^m$. By the definition \eqref{eq:def_approx}, one has $T_{M,m} = \tau^m_M \gamma^{m}$ and $Z^{m}_{T_{M,m}} = Y^m_{\tau^m_M}$. Therefore it is enough to prove that
\begin{align} \label{eq:claim}
 \mathbf{P}^{x} \left [\tau^m_M \leq t, \ell_M \left( Y^{m}_{\tau^m_M} \right) = a \right] =  \mathbf{P}^{y} \left [\tau^m_M \leq t, \ell_M \left( Y^{m}_{\tau^m_M} \right) = a \right], \quad m \in \Zwithneg.
\end{align}
To get this, we consider the paths of $Y^m_k$ starting from $x$ and use the decomposition based on the following collection of stopping times:
$$
\tau_{0} = 0 \qquad \text{and} \qquad
\tau_{i} = \inf \left\{k > \tau_{i-1}: Y^m_k = x \right\} \quad \text{ for }  \quad i>0.
$$
Let $\Delta_M^{(x,i)}, i \in \{1,...,r_1\}$ denote the $M$-complexes with their common vertex $x$ (there are $r_1$ of them as $r_1 = \textrm{rank}(x)$). Then, using the Markov property and symmetry of the process,
\begin{align*}
& \mathbf{P}^{x} \left[ \tau^m_M = k, \ell_M\left( Y^m_{\tau^m_M} \right) = a \right] \\
& =\sum_{i=0}^{\infty} \mathbf{P}^{x} \left[ \tau^m_M = k, \ell_M\left( Y^m_{\tau^m_M} \right)  = a, \tau_i \leq k < \tau_{i+1} \right] \\
& =\sum_{i=0}^{\infty} \sum_{b_0, ..., b_i \in \left\{1, ..., r_1\right\}}
\mathbf{P}^{x} \left[\tau^m_M = k, \ell_M\left( Y^m_{\tau^m_M} \right)  = a, \tau_i \leq k < \tau_{i+1}, Y^m_{\tau_0+1} \in \Delta_M^{(x,b_0)}, ..., Y^m_{\tau_i+1} \in \Delta_M^{(x,b_i)} \right] \\
& =\sum_{i=0}^{\infty} \sum_{b_0, ..., b_i \in \left\{1, ..., r_1\right\}}
\mathbf{P}^{x} \left[\tau^m_M = k, \ell_M\left( Y^m_{\tau^m_M} \right)  = a, \tau_i \leq k < \tau_{i+1}, Y^m_{\tau_0+1} \in \Delta_M^{(x,1)}, ..., Y^m_{\tau_i+1} \in \Delta_M^{(x,1)} \right] \\
& = \sum_{i=0}^{\infty} r_1^{i+1}
\mathbf{P}^{x} \left[\tau^m_M = k, \ell_M\left( Y^m_{\tau^m_M} \right)  = a, \tau_i \leq k < \tau_{i+1}, Y^m_{\tau_0+1} \in \Delta_M^{(x,1)}, ..., Y^m_{\tau_i+1} \in \Delta_M^{(x,1)} \right]  \\
& =\sum_{i=0}^{\infty} r_1^{i+1}
\mathbf{P}^{x} \left[\tau^m_M = k, \ell_M\left( Y^m_{\tau^m_M} \right)  = a, \tau_i \leq k < \tau_{i+1} | Y^m_{\tau_0+1} \in \Delta_M^{(x,1)}, ..., Y^m_{\tau_i+1} \in \Delta_M^{(x,1)} \right] \\
& \ \ \ \ \ \ \ \ \ \ \ \ \ \ \ \ \ \ \ \ \ \ \ \ \ \ \ \
\times \mathbf{P}^{x} \left[Y^m_{\tau_0+1} \in \Delta_M^{(x,1)}, ..., Y^m_{\tau_i+1} \in \Delta_M^{(x,1)} \right]
\end{align*}
Since,
$$
\mathbf{P}^{x} \left[Y^m_{\tau_0+1} \in \Delta_M^{(x,1)}, ..., Y^m_{\tau_i+1} \in \Delta_M^{(x,1)} \right] = \left(\mathbf{P}^{x} \left[Y^m_{1} \in \Delta_M^{(x,1)}\right]\right)^{i+1} = \frac{1}{r_1^{i+1}},
$$
all members under the above sum simplify to
$$
\mathbf{P}^{x} \left[\tau^m_M = k, \ell_M\left( Y^m_{\tau^m_M} \right)  = a, \tau_i \leq k < \tau_{i+1} | Y^m_{\tau_0+1} \in \Delta_M^{(x,1)}, ..., Y^m_{\tau_i+1} \in \Delta_M^{(x,1)} \right], \quad i=0,1,... .
$$
Analogously,
\begin{align*}
\mathbf{P}^{y} & \left[ \tau^m_M = k, \ell_M\left( Y^m_{\tau^m_M} \right)  = a \right] \\
& =\sum_{i=0}^{\infty} \mathbf{P}^{y} \left[\tau^m_M = k, \ell_M\left( Y^m_{\tau^m_M} \right)  = a, \tau_i \leq k < \tau_{i+1} | Y^m_{\tau_0+1} \in \Delta_M^{(y,1)}, ..., Y^m_{\tau_i+1} \in \Delta_M^{(y,1)} \right].
\end{align*}
For every $i=0,1,...$ we have
\begin{gather*}
\mathbf{P}^{x} \left[\tau^m_M = k, \ell_M\left( Y^m_{\tau^m_M} \right)  = a, \tau_i \leq k < \tau_{i+1} | Y^m_{\tau_0+1} \in \Delta_M^{(x,1)}, ..., Y^m_{\tau_i+1} \in \Delta_M^{(x,1)} \right]\\
= \mathbf{P}^{y} \left[\tau^m_M = k, \ell_M\left( Y^m_{\tau^m_M} \right)  = a, \tau_i \leq k < \tau_{i+1} | Y^m_{\tau_0+1} \in \Delta_M^{(y,1)}, ..., Y^m_{\tau_i+1} \in \Delta_M^{(y,1)} \right],
\end{gather*}
so we finally get \eqref{eq:claim} which completes the proof of (1) for $n=1$.

Assume now that for some $n \geq 1$ the assertion holds.  Since no two vertices of the same $M-$complex can  share their labels,
\begin{eqnarray*}
&&\mathbf{P}^{x} \left[T_{M}^{(n+1)}\leq t, \ell_M\left( Z_{T_{M}^{(n+1)}} \right) = a\right]\\
&=& \mathbf{P}^{x} \left[T_{M}^{(n+1)}\leq t,\ell_M\left( Z_{T_{M}^{(n+1)}} \right) = a, T_{M}^{(n)} < t, \ell_M\left( Z_{T_{M}^{(n)}} \right) \neq a \right]\\
 &=& \mathbf{E}^{x}\left[\left.\mathbf{P}^{Z_{T_{M}^{(n)}}} \left[ T_{M}^{(1)}\leq t-u, \ell_M\left(Z_{T_{M}^{(1)}} \right) = a\right]\right|_{u=T_{M}^{(n)}};T_M^{(n)} < t,\ell_M\left( Z_{T_{M}^{(n)}} \right) \neq a\right]
\end{eqnarray*}
Laws of $\big(T_{M}^{(n)}, \ell_M(Z_{T_M^{(n)}})\big)$ are identical under $\mathbf{P}^{x}$ and $\mathbf{P}^{y}$ (inductive assumption). Also, the probability measure under the expectation depends only on  the label $\ell_M\left( Z_{T_{M}^{(n)}} \right),$ not on the actual position of $Z_{T_M^{(n)}}.$ Consequently, $\mathbf{E}^{x}$ can be replaced by $\mathbf{E}^{y}$ and the proof of (1) is concluded.

The proof of (2) is in fact similar to that of the step $n=1$ in part (1). Indeed, if $x,y\notin V_M^{\langle\infty\rangle}$, then we use exactly the same argument. If $x,y\in V_M^{\langle\infty\rangle}$, then we first prove the claimed equality for the random walk by using the same decomposition of paths and by reducing all probabilities under the sums to  proper conditional probabilities. The claimed equality for the Brownian motion $Z$ is then obtained by approximation.

\end{proof}
\begin{theorem}
\label{thm:transition}
Let $x, y \in \mathcal{K}^{\left\langle \infty \right\rangle}$ be two points such that $\pi_{M}(x)=\pi_{M}(y)$. Then the measures {$\pi_M(\mathbf{P}^x)$ and $\pi_M(\mathbf{P}^y)$} on $\left(C\left(\mathbb{R}_{+}, \mathcal{K}^{\left\langle M\right\rangle} \right), \mathcal{B} \left(C\left(\mathbb{R}_{+}, \mathcal{K}^{\left\langle M\right\rangle} \right)\right) \right)$ coincide. Moreover, for every $z \in \mathcal{K}^{\left\langle M \right\rangle}$ we have
\begin{equation}
\label{eq:transition}
g_M(t,x,z) = g_M(t,y,z).
\end{equation}
\end{theorem}

\begin{proof}
Let $x$ and $y$ be as in the assumptions of the theorem. It is enough to prove that the finite-dimensional distributions of underlying measures are identical, i.e. for  $j=1,2...$ and an arbitrary choice of $0 \leq t_1 \leq ... \leq t_{{j}}$ and $\Gamma_1, ..., \Gamma_{{j}} \in \mathcal{B}\left(\mathcal{K}^{\left\langle M\right\rangle}\right)$ we have:
\begin{equation}
\label{eq:finitedist}
\mathbf{P}^{x}\left[Z_{t_1} \in \pi_{M}^{-1}\left(\Gamma_1\right), ..., Z_{t_{{j}}} \in \pi_{M}^{-1}\left(\Gamma_{{j}}\right) \right] = \mathbf{P}^{y}\left[Z_{t_1} \in \pi_{M}^{-1}\left(\Gamma_1\right), ..., Z_{t_{{j}}} \in \pi_{M}^{-1}\left(\Gamma_{{j}}\right) \right].
\end{equation}
We proceed by induction in $j.$

\

\noindent First, let  ${j}=1$ (we drop the subscript '1').

For $t=0$ the equality is self-evident:
\begin{equation}
\mathbf{P}^{x}\left[Z_0 \in \pi_{M}^{-1}\left(\Gamma\right) \right] = \delta_x\left(\pi_{M}^{-1}\left(\Gamma\right) \right)= \delta_y\left(\pi_{M}^{-1}\left(\Gamma\right) \right) = \mathbf{P}^{y}\left[Z_0 \in \pi_{M}^{-1}\left(\Gamma\right) \right]
\end{equation}

\noindent
Let now $t>0$ and consider the following standard decomposition of $C\left(\left[ 0,\infty \right), \mathcal{K}^{\left\langle \infty \right\rangle}\right)$:
\begin{gather}
A_0=\left\{T_{M}^{(1)}>t \right\}, \\
A_n = \left\{T_{M}^{(n)}\leq t < T_{M}^{(n+1)} \right\}, \quad \textrm{for } n \geq 1,
\end{gather}
and for $n = 1, 2, ...$ further
\begin{equation*}
A_n = A_n^{1} \cup ... \cup A_n^{k},
\end{equation*}
where  $A_n^{i}$ indicates that $\ell_M(Z_{T_{M}^{(n)}})=a_i,$  $a_{i} \in \cA = \{a_1,...,a_k\}$, i.e. $A_n^{i} = A_n \cap \left\{\ell_M\left(Z_{T_{M}^{(n)}} \right)= a_i \right\}.$
Consequently,
\begin{eqnarray}\label{eq:th3.1-2}
\mathbf{P}^{x} \left[Z_t \in \pi_{M}^{-1}\left(\Gamma\right) \right] = \mathbf{P}^{x} \left[\left\{Z_t \in \pi_{M}^{-1}\left(\Gamma\right) \right\} \cap A_0\right] + \sum_{n=1}^{\infty} \sum_{i=1}^{k} \mathbf{P}^{x} \left[\left\{Z_t \in \pi_{M}^{-1}\left(\Gamma\right) \right\} \cap A_n^{i}\right].
\nonumber\\
\end{eqnarray}
Now our goal is to show that the terms of the series remain unchanged if we replace $x$ by $y$. By Lemma \ref{lem:lawseq} (2), we get
$$\mathbf{P}^{x} \left[\left\{Z_t \in \pi_{M}^{-1}\left(\Gamma\right) \right\} \cap A_0\right]=\mathbf{P}^{y} \left[\left\{Z_t \in \pi_{M}^{-1}\left(\Gamma\right) \right\} \cap A_0\right]. $$

To get the equality of latter terms in \eqref{eq:th3.1-2}, we use the strong Markov property of $Z_t$. We have
\begin{align*}
\mathbf{P}^{x} \left[\left\{Z_t \in \pi_{M}^{-1}\left(\Gamma\right) \right\}\right. & \left.\cap A_n^{i}\right]
= \mathbf{E}^x \left[ \mathbf{1}_{\left\{T_{M}^{\left(n\right)} \leq t, \ell_M\left(Z_{T_{M}^{\left(n\right)}}\right) =a_i\right\}} \cdot \mathbf{P}^{x}\left[T_{M}^{\left(n+1\right)}>t, Z_{t} \in \pi_{M}^{-1}\left(\Gamma\right) | \mathcal{F}_{T_{M}^{\left(n\right)}}\right] \right] \\
&= \mathbf{E}^x \left[ \mathbf{1}_{\left\{T_{M}^{\left(n\right)} \leq t, \ell_M\left(Z_{T_{M}^{\left(n\right)}}\right) =a_i\right\}} \cdot \mathbf{P}^{Z_{T_{M}^{\left(n\right)}}}\left.\left[T_{M}^{\left(1\right)}>t-s, Z_{t-s} \in \pi_{M}^{-1}\left(\Gamma\right) \right]\right|_{s=T_{M}^{(n)}} \right] \\
&=
\mathbf{E}^x \left[ \mathbf{1}_{\left\{T_{M}^{\left(n\right)} \leq t, \ell_M\left(Z_{T_{M}^{\left(n\right)}}\right) =a_i\right\}} \cdot \mathbf{P}^{v_i}\left.\left[T_{M}^{\left(1\right)}>t-s, Z_{t-s} \in \pi_{M}^{-1}\left(\Gamma\right) \right]\right|_{s=T_{M}^{(n)}}\right]
\\
&= \int_{0}^{t} \mathbf{P}^{v_i}\left[T_{M}^{\left(1\right)}>t-s, Z_{t-s} \in \pi_{M}^{-1}\left(\Gamma\right) \right] d\mu^{x}_{n,i}\left(s\right).
\end{align*}
In this formula, $v_i$ is the unique vertex of $\mathcal K^{\langle M \rangle}$ with label $a_i$ and $\mu^{x}_{n,i}$ is the distribution of $T_{M}^{\left(n\right)}\mathbf 1_{\{\ell_M(Z_{T_M^{(n)}})=a_i\}}$ under $\mathbf{P}^{x}$. The equality of the second and third line above follows from Lemma \ref{lem:lawseq} (2) (i.e. for fixed $s$ the probability under the expectation depends only on the label $\ell_M\left( Z_{T_{M}^{(n)}} \right),$ not on the actual position of $Z_{T_M^{(n)}}$). Finally, from Lemma \ref{lem:lawseq} (1) we get $\mu^{x}_{n,i} = \mu^{y}_{n,i}$ and, therefore,
\begin{equation*}
\mathbf{P}^{x} \left[\left\{Z_t \in \pi_{M}^{-1}\left(\Gamma\right) \right\} \cap A_n^{i}\right] = \mathbf{P}^{y} \left[\left\{Z_t \in \pi_{M}^{-1}\left(\Gamma\right) \right\} \cap A_n^{i}\right].
\end{equation*}
This completes the proof of \eqref{eq:finitedist} for ${j}=1$. In particular, for every $\Gamma \in \mathcal{B} \left(\mathcal{K}^{\left\langle M \right\rangle } \right)$, one has
\begin{equation*}
\int_{\Gamma} g_{M}(t,x,z) d\mu(z) = \mathbf{P}^{x} \left[ Z_t \in \pi_{M}^{-1} \left(\Gamma\right)\right] = \mathbf{P}^{y} \left[ Z_t \in \pi_{M}^{-1} \left(\Gamma\right)\right] = \int_{\Gamma} g_{M}(t,y,z) d\mu(z).
\end{equation*}
Since $g_M$ is continuous in $z$ (Lemma \ref{lem:properties}), we infer that $g_M\left(t,x,z\right) = g_M \left(t,y,z\right)$ for all $z \in \mathcal{K}^{\left\langle M \right\rangle }$. This gives \eqref{eq:transition}.

We can now complete the inductive proof of \eqref{eq:finitedist}.
Assume that the assertion (\ref{eq:finitedist}) holds for some ${j}\geq 1$, arbitrary choice of $t_1, ..., t_{{j}}\geq 0$ and $\Gamma_1, ..., \Gamma_{{j}}\in\mathcal B(\mathcal K^{(M)})$. Now, let $0 \leq t_1 \leq ... \leq t_{{j}+1}$ and $\Gamma_1, ..., \Gamma_{{j}+1} \in \mathcal{B} \left(\mathcal{K}^{\left\langle M \right\rangle } \right)$ be arbitrary. By the Markov property of $Z$, the properties of the map $\pi_M$ and Fubini-Tonelli theorem, we have
\begin{align*}
\mathbf{P}^{x} & \left[Z_{t_1} \in \pi_{M}^{-1} \left(\Gamma_1\right), ... ,Z_{t_{{j}+1}} \in \pi_{M}^{-1} \left(\Gamma_{{j}+1}\right) \right] \\
& = \int_{\pi_{M}^{-1} \left(\Gamma_1\right) } g\left(t_1, x, z\right) \mathbf{P}^{z} \left[ Z_{t_2-t_1} \in \pi_{M}^{-1} \left(\Gamma_{2}\right), ... ,Z_{t_{{j}+1}-t_1} \in \pi_{M}^{-1} \left(\Gamma_{{j}+1}\right)\right] d\mu\left(z\right) \\
& = \int_{\Gamma_1} \sum_{z^{\prime} \in \pi_{M}^{-1}(z)} g\left(t_1, x, z^{\prime}\right) \mathbf{P}^{z^{\prime}} \left[ Z_{t_2-t_1} \in \pi_{M}^{-1} \left(\Gamma_2\right), ... ,Z_{t_{{j}+1}-t_1} \in \pi_{M}^{-1} \left(\Gamma_{{j}+1}\right)\right] d\mu\left(z\right).
\end{align*}
From the inductive assumption we can now replace the measure $\mathbf{P}^{z^{\prime}}$ under the integral with $\mathbf{P}^z$ and then, from the already shown identity \eqref{eq:transition}, we get that
\begin{equation*}
\sum_{z^{\prime} \in \pi_{M}^{-1}(z)} g\left(t_1, x, z^{\prime}\right) = \sum_{z^{\prime} \in \pi_{M}^{-1}(z)} g\left(t_1, y, z^{\prime}\right), \quad \mu\text{-a.a.} \, z \in \Gamma_1,
\end{equation*}
After these rearrangements, we can now turn the formula back to the initial form, but with $x$ replaced with $y$. The theorem follows.
\end{proof}

Next we show the Chapman-Kolmogorov identity for the kernels $g_{M}(t,x,y)$.
\begin{lemma} \label{lem:ChK}
For $t,s >0$, $x,z \in \mathcal{K}^{\left\langle M \right\rangle}$
\begin{equation}
g_{M}(t+s, x,z) = \int_{\mathcal{K}^{\left\langle M \right\rangle}} g_{M}(t,x,y) g_{M} (s,y,z) d\mu(y).
\end{equation}
\end{lemma}
\begin{proof}
Suppose that $z \in V_{M}^{\left\langle M \right\rangle}$. By the Chapman-Kolmogorov equation for the density $g(t,x,y)$,  Fubini-Tonelli theorem and the properties of $\pi_M$, we have
\begin{align*}
g_{M}(t+s,x,z) & =\sum_{z'\in \pi_{M}^{-1} (z)}{g(t+s,x,z')} \textrm{rank}(z') = \sum_{z'\in \pi_{M}^{-1} (z)} \int_{\mathcal{K}^{\left\langle \infty \right\rangle}} g(t,x,y) g(s,y,z') \textrm{rank}(z') d\mu(y)\\
& =\int_{\mathcal{K}^{\left\langle \infty \right\rangle}} g(t,x,y) \sum_{z'\in \pi_{M}^{-1} (z)}g(s,y,z')\textrm{rank}(z') d\mu(y)
= \int_{\mathcal{K}^{\left\langle \infty \right\rangle}} g(t,x,y) g_M(s,y,z) d\mu(y) \\
& = \int_{\mathcal{K}^{\left\langle M \right\rangle}} \sum_{y^{\prime} \in \pi_M^{-1}(y)} g(t,x,y^{\prime}) g_M(s,y^{\prime},z) d\mu(y).
\end{align*}
Now, by \eqref{eq:transition}, we can replace $g_M(s,y^{\prime},z)$ with $g_M(s,y,z)$
and observe that for nonvertex $y\in\mathcal K^{\langle M\rangle}$ $\sum_{y'}g(t,x,y')=g_M(t,x,y).$ This gives the conclusion. For $z \notin V_{M}^{\left\langle M \right\rangle}$ the proof is the same (we just drop `$\textrm{rank}(z')$').
\end{proof}

We are now in a position to show that $g_{M}$ is symmetric in $x,y \in \mathcal{K}^{\left\langle M \right\rangle}$.
\begin{lemma} \label{lem:symmetry}
For every $t>0$ and $x,y \in \mathcal{K}^{\left\langle M \right\rangle}$ one has $g_{M} (t,x,y)=g_{M}(t,y,x)$.
\end{lemma}

\begin{proof}
We will prove that for $x,y \in \mathcal{K}^{\left\langle M \right\rangle} \backslash V_{M}^{\left\langle M \right\rangle}$ we have:
\begin{equation}
\label{eq:sym2}
g_{M} (t,x,y) = \lim_{n \to \infty} \frac{\mu\left(\mathcal{K}^{\left\langle M \right\rangle}\right)}{\mu\left(\mathcal{K}^{\left\langle n \right\rangle}\right)} \sum_{A \left(n,x,y\right)} g\left(t, x', y'\right),
\end{equation}
where $A\left(n,x,y\right) = \left\{\left(x',y'\right): x' \in \pi_{M}^{-1}(x), y' \in \pi_{M}^{-1}(y), x', y' \in \mathcal{K}^{\left\langle n \right\rangle} \right\}$. As \eqref{eq:sym2} is symmetric in $x$ and $y$, it proves the assertion of the lemma in the case when $x$ and $y$ are not vertices from $V_{M}^{\left\langle M \right\rangle}$.

To see \eqref{eq:sym2} we make use of \eqref{eq:transition}. The value of the sum
\begin{equation*}
\sum_{y' \in \pi_{M}^{-1}(y)} g(t,x,y')
\end{equation*}
does not depend on the particular choice of $x$ within the same fiber. For $n \geq M$ we have
\begin{eqnarray*}
g_{M}(t,x,y) &= &\frac{\mu\left(\mathcal{K}^{\left\langle M \right\rangle}\right)}{\mu\left(\mathcal{K}^{\left\langle n \right\rangle}\right)} \sum_{B \left(n,x,y\right)} g\left(t, x', y'\right)\\
&=& \frac{\mu\left(\mathcal{K}^{\left\langle M \right\rangle}\right)}{\mu\left(\mathcal{K}^{\left\langle n \right\rangle}\right)} \sum_{A \left(n,x,y\right)} g\left(t, x', y'\right) + \frac{\mu\left(\mathcal{K}^{\left\langle M \right\rangle}\right)}{\mu\left(\mathcal{K}^{\left\langle n \right\rangle}\right)} \sum_{C \left(n,x,y\right)} g\left(t, x', y'\right) =: \alpha_n + \beta_n,
\end{eqnarray*}
where
\begin{eqnarray*}
A(n,x,y) && \textrm{ is defined above,}\\
B(n,x,y) &=& \left\{\left(x',y'\right): x' \in \pi_{M}^{-1}(x), y' \in \pi_{M}^{-1}(y), x' \in \mathcal{K}^{\left\langle n \right\rangle} \right\},\\
C(n,x,y) &=& \left\{\left(x',y'\right): x' \in \pi_{M}^{-1}(x), y' \in \pi_{M}^{-1}(y), x' \in \mathcal{K}^{\left\langle n \right\rangle}, y' \notin  \mathcal{K}^{\left\langle n \right\rangle} \right\}.
\end{eqnarray*}
{To justify \eqref{eq:sym2}, it} suffices to show that $\beta_n$ goes to zero as $n \to \infty$. Let us assume that $n$ is large enough such that $\left\lfloor\log_N n \right\rfloor>M$. Then we have

\begin{eqnarray}
\label{eq:beta}
\nonumber
\beta_n &= &\frac{\mu\left(\mathcal{K}^{\left\langle M \right\rangle}\right)}{\mu\left(\mathcal{K}^{\left\langle n \right\rangle}\right)} \sum_{x'\in D(n,x) } \left(\sum_{\substack{y^{\prime} \in \pi_{M}^{-1}(y) \\ y' \notin \mathcal{K}^{\left\langle n \right\rangle}}}  g\left(t, x', y'\right)\right)+ \frac{\mu\left(\mathcal{K}^{\left\langle M \right\rangle}\right)}{\mu\left(\mathcal{K}^{\left\langle n \right\rangle}\right)} \sum_{{x^{\prime} \in E(n,x)}} \left(\sum_{\substack{y^{\prime} \in \pi_{M}^{-1}(y) \\ y' \notin \mathcal{K}^{\left\langle n \right\rangle}} }  g\left(t, x', y'\right)\right)\\
 &=&: \beta_{n,1} + \beta_{n,2}
\end{eqnarray}
where
\begin{equation*}
D(n,x) = \pi_{M}^{-1}(x) \cap \left\{x' \in \mathcal{K}^{\left\langle n \right\rangle}: V\left(  \mathcal{K}^{\left\langle n \right\rangle}\right) \cap \Delta_{\left\lfloor\log_N n \right\rfloor}(x') \neq \emptyset \right\}
\end{equation*}
is the set of those $x' \in \pi_{M}^{-1}(x)$ which are close to the vertices of $\mathcal{K}^{\left\langle n \right\rangle}$ and
$$ E(n,x) = \pi_{M}^{-1}(x) \cap \mathcal{K}^{\left\langle n \right\rangle} \cap D^c(n,x)$$
is the set of those $x'$ that are far from all the vertices.
First note that by \eqref{eq:transition} one has
$$
\beta_{n,1} \leq \frac{\mu\left(\mathcal{K}^{\left\langle M \right\rangle}\right)}{\mu\left(\mathcal{K}^{\left\langle n \right\rangle}\right)} \cdot \# D(n,x) \cdot \sup_{x, y \in \mathcal{K}^{\left\langle M \right\rangle}} q_{M}(t,x,y).
$$
Recall that, for any $n,$ the set $V^{\langle n \rangle}$ has exactly $k$ vertices.
Now, since the cardinality of $D(n,x)$ is the number of $M$-complexes within the $k$ $\left\lfloor\log_N n \right\rfloor$-complexes (each $\left\lfloor\log_N n \right\rfloor$-complex is adjacent to one of the $k$ vertices in $V\left(  \mathcal{K}^{\left\langle n \right\rangle}\right)$), $\beta_{n,1}$ can be estimated as follows
$$
\beta_{n,1}
\leq \frac{\mu\left(\mathcal{K}^{\left\langle M \right\rangle}\right)}{\mu\left(\mathcal{K}^{\left\langle n \right\rangle}\right)} \cdot \frac{k \mu\left(\mathcal{K}^{\left\langle \left\lfloor\log_N n \right\rfloor \right\rangle}\right)}{\mu\left(\mathcal{K}^{\left\langle M \right\rangle}\right)} \cdot c_1 = c_1 \frac{k N^{\left\lfloor\log_N n \right\rfloor}}{N^n} \leq c_1 \frac{k n}{N^n},
$$
where $c_1 = c_1(t,M):=\sup_{x, y \in \mathcal{K}^{\left\langle M \right\rangle}} g_{M}(t,x,y)$. This gives that $\beta_{n,1} \to 0$ as $n \to \infty$.

To estimate $\beta_{n,2}$ we notice that in this case $x'$ and $y'$ are far away. If $m > M$ and $d_m(x',y')>2$, then $d_{M}(x',y')>2^{m-M} +2$. By using this with
 $m=\lfloor \log_N{n}\rfloor$, together with the estimate $\# E(n,x)\leq
 \frac{\mu(\mathcal K^{\langle n\rangle})}{\mu(\mathcal K^{\langle N\rangle})},$  we get
\begin{align*}
\beta_{n,2} & \leq \frac{\mu\left(\mathcal{K}^{\left\langle M \right\rangle}\right)}{\mu\left(\mathcal{K}^{\left\langle n \right\rangle}\right)} \cdot \# E(n,x) \cdot \sup_{x'\in E(n,x)} \sum_{\substack{y' \in \pi_{M}^{-1}(y) \\ y' \notin \mathcal{K}^{\left\langle n \right\rangle}}}  g\left(t, x', y'\right)\\
& \leq \frac{\mu\left(\mathcal{K}^{\left\langle M \right\rangle}\right)}{\mu\left(\mathcal{K}^{\left\langle n \right\rangle}\right)} \cdot \frac{\mu\left(\mathcal{K}^{\left\langle n \right\rangle}\right)}{\mu\left(\mathcal{K}^{\left\langle M \right\rangle}\right)} \cdot \sup_{x'\in E(n,x)}  \sum_{\left\{y' \in \pi_{M}^{-1}(y) \, : d_{\left\lfloor\log_N n \right\rfloor}(x', y')>2\right\}}  g\left(t, x', y'\right)\\
& \leq \sup_{x'\in E(n,x)} \sum_{\left\{y' \in \pi_{M}^{-1}(y) \, : d_{M}(x', y')>2^{\left\lfloor\log_N n \right\rfloor-M}+2\right\}}  g\left(t, x', y'\right) \\
& =\sup_{x'\in E(n,x)}  \sum_{j=2^{\left\lfloor \log_N n\right\rfloor-M}+2}^{\infty} \sum_{\left\{y' \in \pi_{M}^{-1}(y) \, : d_{M}(x', y')=j\right\}}  g\left(t, x', y'\right)\\
& \leq \sum_{j=2^{\left\lfloor \log_N n\right\rfloor-M}+2}^{\infty}\sup_{x'\in E(n,x)}  \sum_{\left\{y' \in \pi_{M}^{-1}(y) \, : d_{M}\left(x',y'\right) =j \right\} }  g\left(t, x', y'\right).
\end{align*}
If $d_{M}(x', y')=j>2$, then $|x'-y'|\geq c_2 (M) j^{1/d_f}$ (Lemma \ref{lem:metrics}).
Then, by applying the upper bound in \eqref{eq:kum}, and finally Lemma \ref{lem:complab} (to estimate the number of  points $y'$ under the inner sum), the above estimate can be continued as follows
\begin{align*}
& \sum_{j=2^{\left\lfloor \log_N n\right\rfloor-M}+2}^{\infty}\sup_{x'\in E(n,x)}  \sum_{\left\{y' \in \pi_{M}^{-1}(y) \, : d_{M}\left(x',y'\right) =j \right\} } g\left(t, x', y'\right)\\
& \leq c_3 \sum_{j=2^{\left\lfloor \log_N n\right\rfloor-M}+2}^{\infty}\sup_{x'\in E(n,x)}  \sum_{\left\{y' \in \pi_{M}^{-1}(y) \, : d_{M}\left(x',y'\right) =j \right\} } t^{-d_s/2} \exp \left(-c_4 \left(\frac{\left|x'-y' \right|^{d_w}}{t} \right)^{\frac{1}{d_J -1}} \right)\\
& \leq c_3 \sum_{j=2^{\left\lfloor \log_N n\right\rfloor-M}+2}^{\infty} \sup_{x'\in E(n,x)}  \# \left\{y' \in \pi_{M}^{-1}(y) \, : d_{M}\left(x',y'\right) =j \right\}\cdot   t^{-d_s/2} \exp \left(-c_5 \left(\frac{j ^{d_w/d_f}}{t} \right)^{\frac{1}{d_J -1}} \right)\\
& \leq c_6 \sum_{j=2^{\left\lfloor \log_N n\right\rfloor-M}+2}^{\infty} j^{d_f} t^{-d_s/2} \exp \left(-c_5 \left(\frac{j^{d_w / d_f}}{t} \right)^{1 / \left(d_J -1 \right)} \right),
\end{align*}
where the constants $c_3,...,c_6$ do not depend on $n$. We then see that $\beta_{n,2}$ is dominated by the tail of a convergent series - hence $\beta_{n,2} \to 0$ as $n \to \infty$. This completes the proof for $x,y \in \mathcal{K}^{\left\langle M \right\rangle} \backslash V_{M}^{\left\langle M \right\rangle}$. For arbitrary $x$ and $y$ the assertion of the lemma follows by continuity.
\end{proof}

We are now ready to give a formal proof of Theorem \ref{thm:main1}.

\begin{proof}[Proof of Theorem \ref{thm:main1}]
First note that the first part of assertion (1) (continuity and boundedness of $g_M(t,x,y)$) and the assertions (3) and (4) have already been proven in Lemmas \ref{lem:properties}, \ref{lem:ChK}, \ref{lem:symmetry}, respectively. 
Moreover, the inclusion $T_t^M \big(L^{\infty}(\mathcal{K}^{\left\langle M\right\rangle})\big) \subset C_b(\mathcal{K}^{\left\langle M\right\rangle})$, $t>0$, completing (1), follows directly from the continuity and boundedness of the kernel $g_M(t,x,y)$ by the Lebesgue dominated convergence theorem (recall that $\mu(\mathcal{K}^{\left\langle M\right\rangle})< \infty$). Therefore, it suffices to show the strong continuity in the assertion (2). By \eqref{eq:fin_us} we have for $f \in C_b(\mathcal{K}^{\left\langle M\right\rangle})$
\begin{align*}
& \int_{\mathcal{K}^{\left\langle M\right\rangle}} f(y) g_M(t,x,y) \mu(dy)- f(x)  \\ & = \int_{\mathcal{K}^{\left\langle M\right\rangle}} (f(y) - f(x)) g_M(t,x,y) \mu(dy) \\
& = \int_{\mathcal{K}^{\left\langle M\right\rangle}} (f(y) - f(x)) \left(\sum_{y'\in \pi_{M}^{-1} (y)}{g(t,x,y')}\right)\mu(dy) \\
& = \int_{\mathcal{K}^{\left\langle M\right\rangle}} (f(y) - f(x)) \left(\sum_{n=1}^2 \sum_{\Delta_M\left(y'\right) \in \mathcal{L}_{M,n,x}} {g(t,x,y')}\right)\mu(dy) + \int_{\mathcal{K}^{\left\langle M\right\rangle}} (f(y) - f(x)) \left(\sum_{n=3}^{\infty} a_{n,t,x}\right)\mu(dy) \\
& =: I_1(t,x) + I_2(t,x).
\end{align*}
By  the properties of the projection $\pi_M$ we get
\begin{align*}
I_1(t,x) & = \sum_{n=1}^2 \sum_{\Delta \in \mathcal{L}_{M,n,x}} \int_{\Delta} (f_M(y) - f_M(x)) g(t,x,y)\mu(dy) \\ & = \int_{\bigcup_{\Delta \in \mathcal{L}_{M,1,x} \cup \mathcal{L}_{M,2,x}}} (f_M(y) - f_M(x)) g(t,x,y)\mu(dy),
\end{align*}
where $f_M: \mathcal{K}^{\left\langle \infty \right\rangle} \to \R$ is defined by $f_M(y):=f(\pi_M(y))$, $y \in \mathcal{K}^{\left\langle \infty \right\rangle}$. Observe that $f_M \in C_b(\mathcal{K}^{\left\langle \infty\right\rangle})$
(uniformly continuous in fact), in particular for any given $\epsilon>0$ there is an $\eta>0$ such that $|f_M(x)-f_M(y)|\leq\epsilon$ once $d_M(x,y)\leq \eta.$
From this 
we can derive that
\begin{eqnarray*}
\sup_{x \in \mathcal{K}^{\left\langle M\right\rangle}} |I_1(t,x)| &\leq& \sup_{x \in \mathcal{K}^{\left\langle M\right\rangle}} \int_{\bigcup_{\Delta \in \mathcal{L}_{M,1,x} \cup \mathcal{L}_{M,2,x}}} |f_M(y) - f_M(x)| g(t,x,y)\mu(dy)\\
 &\leq & \sup_{x\in \mathcal K^{\langle M\rangle}} \int_{B(x,\eta)}+\int_
 {\left(\bigcup_{\Delta \in \mathcal{L}_{M,1,x} \cup \mathcal{L}_{M,2,x}}\right)\cap B(x,\eta)^c}
 \\
&\leq& \epsilon + 2\|f_M\|_\infty \sup_{x\in\mathcal K^{\langle M\rangle}} \mathbf{P}(Z_t\in B(x,\delta)^c).
\end{eqnarray*}
Letting first $t\to 0$ and then $\epsilon\to 0$ we get the assertion.
Moreover, by \eqref{eq:for_later_use} there exists $\delta>0$ such that $|\sum_{n=3}^{\infty} a_{n,t,x}| \leq c_1 t^{-\delta}$, uniformly in $x$, and then
$$
\sup_{x \in \mathcal{K}^{\left\langle M\right\rangle}} |I_2(t,x)| \leq 2 c_1 \left\|f\right\|_{\infty} \mu(\mathcal{K}^{\left\langle M\right\rangle}) t^{-\delta} \to 0 \quad \text{as} \ t \to 0^{+}.
$$
The assertion (2) follows. This completes the proof of the theorem.
\end{proof}

\appendix
\section{}
\label{sec:app}

We now discuss in detail the relation of the $M$-graph distance (introduced in \ref{def:graphmetric}) to the Euclidean distance on simple nested fractal. The following facts were used in proofs in the previous section.

For $E,F\subset\mathcal K^{\langle\infty\rangle}$ closed and bounded,   $\mbox{dist}(E,F)=\inf\{|x-y|:x\in E, y\in F\},$ denotes their Euclidean distance.
\begin{lemma}\label{lem:const}
We have
\[\inf\left\{\mbox{\rm dist}(\Delta_0^{(1)},\Delta_0^{(2)}): \Delta_0^{(1)},\Delta_0^{(2)}\in\mathcal T_0,\Delta_0^{(1)}\cap \Delta_0^{(2)}=\emptyset\right\}=:C_5>0.\]
\end{lemma}

\begin{proof}
Let
\begin{equation}
C_5^{(m)} = \inf \left\{ \mbox{\rm dist}(\Delta_0^{(1)},\Delta_0^{(2)}):
\Delta_0^{(1)},\Delta_0^{(2)}\in\mathcal T_0,\Delta_0^{(1)},\Delta_0^{(2)} \subset \mathcal{K}^{\left\langle m \right\rangle}, \Delta_0^{(1)}\cap \Delta_0^{(2)}=\emptyset\right\}>0.
\end{equation}
We will see by induction that $C_5^{(m)}=C_5^{(2)},$ $m=2,3,....$

 For $m=2$ there is nothing to prove. Assume that $C_5^{(m)}=C_5^{(2)}$ for some $m>2,$ take {$\Delta_0^{(1)},\Delta_0^{(2)}\subset \mathcal K^{\langle m+1\rangle}.$}
 We have three possibilities.

\smallskip

\textsc{Case 1.} Both $\Delta_0^{(1)},\Delta_0^{(2)}$ are subsets of a common $m-$complex. This complex is an isometric image of $\mathcal K^{\langle m\rangle},$ thus ${\mbox{dist}\left(\Delta_0^{(1)},\Delta_0^{(2)}\right)\geq C_5^{(m)}= C_5^{(2)}}$ by assumption.

\smallskip

\textsc{Case 2.} $\Delta_0^{(1)}\subset\widetilde{\Delta}_m^{(1)}$, $\Delta_0^{(2)}\subset\widetilde{\Delta}_m^{(2)}$, where $\widetilde{\Delta}_m^{(1)},\widetilde{\Delta}_m^{(2)}$ are two disjoint $m-$complexes.
 Then $\frac{1}{L}\widetilde{\Delta}_m^{(1)},$ $\frac{1}{L}\widetilde{\Delta}_m^{(2)}$ are two disjoint $(m-1)-$complexes included in $\mathcal K^{\langle m \rangle};$ from the assumption the distance from any $0-$complex in $\frac{1}{L}\widetilde{\Delta}_m^{(1)}$ to any $0-$complex in $\frac{1}{L}\widetilde{\Delta}_m^{(2)}$ is not smaller than $C_{5}^{(m)}.$ Then
  $\frac{1}{L}\Delta_0^{(1)}, $ $\frac{1}{L}\Delta_0^{(2)}$ are two $(-1)$-complexes included in some (disjoint) $0-$complexes from $\mathcal K^{\langle m \rangle},$ consequently
$$\mbox{dist}\left(\Delta_0^{(1)}, \Delta_0^{(2)}\right)= L \ \mbox{dist}\left(\frac{1}{L}\Delta_0^{(1)}, \frac{1}{L}\Delta_0^{(2)}\right)\geq LC_5^{(m)}> C_5^{(m)}=C_5^{(2)}.$$

\smallskip

\textsc{Case 3.} {$\Delta_0^{(1)}\subset\widetilde{\Delta}_m^{(1)}$, $\Delta_0^{(2)}\subset\widetilde{\Delta}_m^{(2)}$, where $\widetilde{\Delta}_m^{(1)},\widetilde{\Delta}_m^{(2)}$ are two neighboring $m-$complexes. Let $\widetilde{\Delta}_m^{(1)}\cap\widetilde{\Delta}_m^{(2)}=\{z\}, z \in V_m^{\langle \infty \rangle}.$}

\smallskip

(a) If   $\Delta_0^{(1)}\subset \widetilde{\Delta}_{m-1}^{(1)},$
$\Delta_0^{(2)}\subset\widetilde{\Delta}_{m-1}^{(2)},$ where $\widetilde{\Delta}_{m-1}^{(1)},\widetilde{\Delta}_{m-1}^{(2)}$ are two disjoint $(m-1)-$complexes. Then, as above, $\frac{1}{L}\widetilde{\Delta}_{m-1}^{(1)},$ $\frac{1}{L}\widetilde{\Delta}_{m-1}^{(2)}$ are two disjoint $(m-2)-$complexes included in $\mathcal K^{\langle m \rangle}$ and  $\frac{1}{L}\Delta_0^{(1)}, $ $\frac{1}{L}\Delta_0^{(2)}$ are two $(-1)$-complexes included in some (disjoint) $0-$complexes from $\mathcal K^{\langle m \rangle},$ consequently
$$\mbox{dist}\left(\Delta_0^{(1)}, \Delta_0^{(2)}\right)= L \ \mbox{dist}\left(\frac{1}{L}\Delta_0^{(1)}, \frac{1}{L}\Delta_0^{(2)}\right)\geq LC_5^{(m)}> C_5^{(m)}=C_5^{(2)}.$$

\smallskip

 (b) The remaining case is: $\Delta_0^{(1)}\subset \widetilde{\Delta}_{m-1}^{(1)},$
$\Delta_0^{(2)}\subset \widetilde{\Delta}_{m-1}^{(2)},$ and $\widetilde{\Delta}_{m-1}^{(1)}, \widetilde{\Delta}_{m-1}^{(2)}$ are two neighboring $(m-1)-$complexes. Then we necessarily have $\widetilde{\Delta}_{m-1}^{(1)}\cap\widetilde{\Delta}_{m-1}^{(2)}=\{z\}.$ From scaling, $z':=\frac{1}{L}z\in V_{m-1}^{\langle\infty\rangle}\cap\mathcal K^{\langle m \rangle}$.
Then $\widetilde{\Delta}_{m-1}^{(3)}:=\left(\widetilde{\Delta}_{m-1}^{(1)}-z\right)+\frac{1}{L}z$
and
 $\widetilde{\Delta}_{m-1}^{(4)}:=\left(\widetilde{\Delta}_{m-1}^{(2)}-z\right)+\frac{1}{L}z$
 are two $(m-1)-$complexes included in $\mathcal K^{\langle m \rangle}$ with common vertex $z'.$ Therefore
 $\widetilde{\Delta}_0^{(1)}:=\left(\Delta_0^{(1)}-z\right)+\frac{1}{L}z$ and
 $\widetilde{\Delta}_0^{(2)}:=\left(\Delta_0^{(2)}-z\right)+\frac{1}{L}z$
 are two $0-$complexes included in $\mathcal K^{\langle m\rangle}$ for which $\mbox{dist}\left(\Delta_0^{(1)}, \Delta_0^{(2)}\right) = \mbox{dist}\left(\widetilde{\Delta}_0^{(1)}, \widetilde{\Delta}_0^{(2)}\right).$ From the assumption we have $\mbox{dist}\left(\widetilde{\Delta}_0^{(1)}, \widetilde{\Delta}_0^{(2)}\right) \geq C_5^{(m)}=C_5^{(2)}$.

\begin{figure}[ht]
\centering
\includegraphics[scale=0.2]{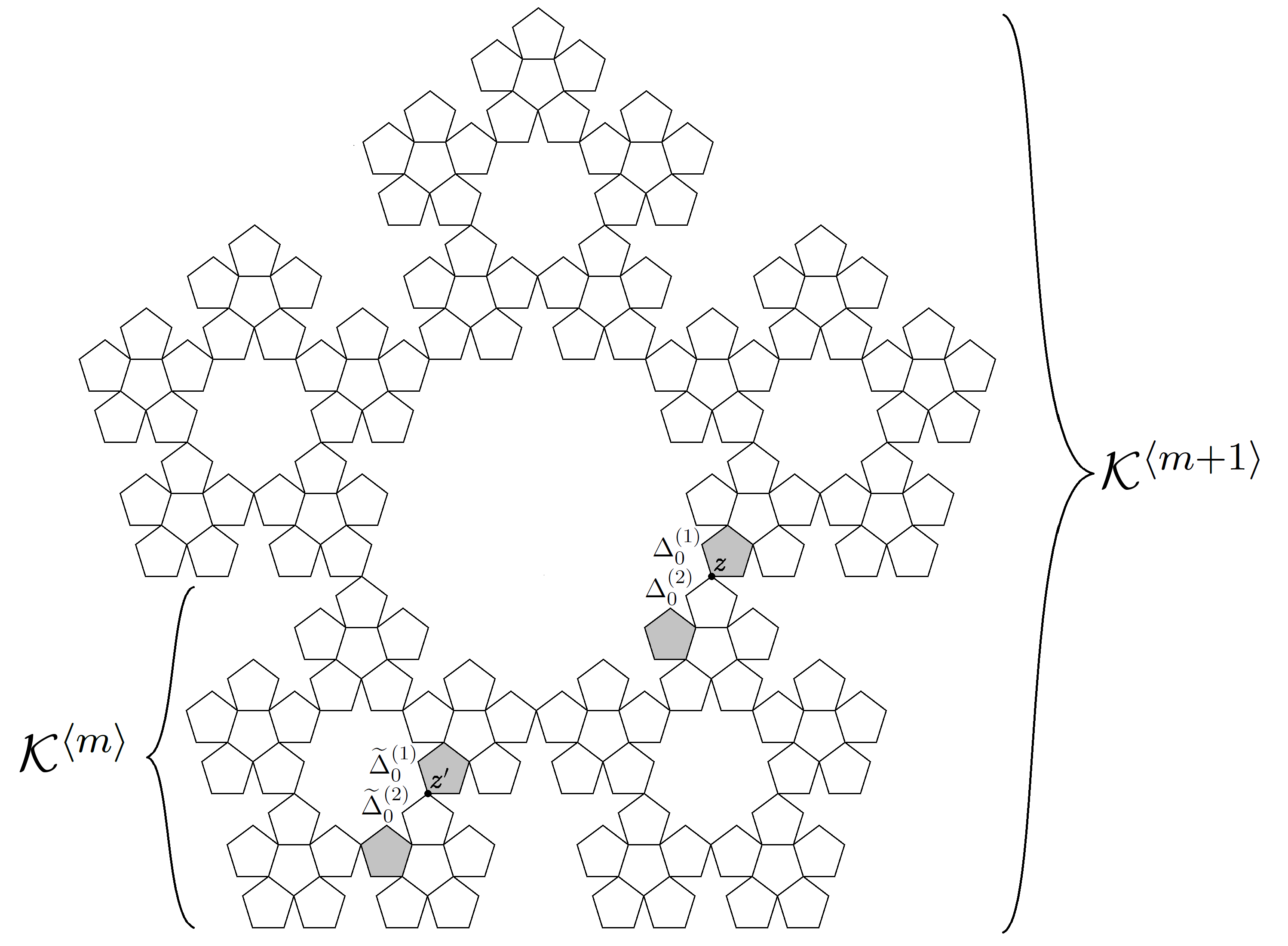}
\caption{The case (3) (b) in the proof of Lemma \ref{lem:const}: $\textrm{dist}\left(\widetilde{\Delta}_0^{(1)}, \widetilde{\Delta}_0^{(2)}\right) = \textrm{dist}\left(\Delta_0^{(1)}, \Delta_0^{(2)}\right)$.}
\label{fig:complexshift}
\end{figure}

The common value of $C_5^{(m)}$ will be denoted by $C_5.$ The lemma follows.
\end{proof}

It will be convenient to have a statement for points instead of complexes.

\begin{corollary}
\label{cor:const} For all $x,y \in \mathcal{K}^{\left\langle \infty \right\rangle} \backslash V^{\left\langle \infty \right\rangle}_{0}$ satisfying $\Delta_0(x) \cap \Delta_0(y) = \emptyset$, one has $|x-y| \geq C_5.$ where $C_5$ is the constant from Lemma \ref{lem:const}.
\end{corollary}

\begin{lemma}
\label{lem:metrics}
For every $M \in \mathbb{Z}$ there exist positive constants $C_6(M)$, $C_7(M)$ such that for every $x,y \in \mathcal{K}^{\left\langle \infty \right\rangle}$ we have
\begin{equation} \label{eq:metrics1}
C_{6}(M) \left|x-y\right| \leq d_M\left(x,y\right) \leq \max \left\{2, C_7(M) \left|x-y\right|^{d_f} \right\},
\end{equation}
where $d_f = \frac{\log N}{\log L}$ is the Hausdorff dimension of the fractal.
\end{lemma}

\begin{proof}
The first inequality comes from the fact that $d_M$ is the metric counting the number of complexes we must visit when passing from $x$ to $y$. By the triangle inequality, the length of the line segment joining $x$ and $y$ is not smaller than the sum of lengths of polygonal chain segments,
\begin{equation*}
\left|x-y\right| \leq d_{M}\left(x,y\right) \cdot \textrm{diam}\left(\Delta_M\right),
\end{equation*}
where $\textrm{diam}\left(\Delta_M\right)$ is the common diameter of any $M$-complex. So the leftmost inequality  in \eqref{eq:metrics1} holds with $C_6 (M)= \frac{1}{\textrm{diam}\left(\Delta_M\right)} = \frac{1}{L^M \textrm{diam}\left(\Delta_0\right)}$.

{Now we  prove the rightmost inequality. If $d_M(x,y)=0$ of $d_M(x,y)=1$ then the inequality is obvious.}  Take $x,y \in \mathcal{K}^{\left\langle \infty \right\rangle} \backslash V^{\left\langle \infty \right\rangle}_{M}$ with $d_M(x,y)>1$, i.e. $\Delta_M(x)\neq \Delta_M(y)$. Let $M_B>M$ be the smallest number such that $\Delta_{M_B}(x) = \Delta_{M_B}(y)$. One of the following three cases occurs.

\textsc{Case 1.}
 If $\Delta_{M_B-1}(x) \cap \Delta_{M_B-1}(y) = \emptyset$, then
\begin{equation}
\label{eq:complexest1}
d_{M}\left(x,y\right) \leq N^{M_B-M},
\end{equation}
where $N$ is a number of similitudes, i.e. $N^{M_B-M}$ is a number of $M$-complexes in any $M_B$-complex. On the other hand, from scaling,
\begin{equation}
\label{eq:euclidest1}
\left|x-y\right| \geq C_5 \cdot L^{M_B-1},
\end{equation}
where $L$ is the length scaling factor of the fractal and $C_5$ is the minimum of distances between two disjoint $0$-complexes, introduced in Lemma \ref{lem:const}.
Consequently,
\begin{equation*}
d_M(x,y) \leq C_5^{- \log N / \log L} N^{-M+1} \left|x-y\right|^{\log N / \log L}.
\end{equation*}
Since $\frac{\log N}{\log L}=d_f,$ the inequality follows.

\smallskip

 \textsc{Case 2.} If $\Delta_{M_B-1}(x) \cap \Delta_{M_B-1}(y) \neq \emptyset$ and there exists $M_S > M$ such that $\Delta_{M_S}(x) \cap \Delta_{M_S}(y) \neq \emptyset$, but $\Delta_{M_S-1}(x) \cap \Delta_{M_S-1}(y) = \emptyset$, then, similarly as above,
\begin{equation*}
d_M(x,y) \leq 2 N^{M_S-M},
\end{equation*}
\begin{equation*}
\left|x-y \right| \geq C_5 \cdot L^{M_S-1}
\end{equation*}
so that
\begin{equation*}
d_M(x,y) \leq 2 C_5^{- \log N / \log L} N^{-M+1} \left|x-y\right|^{\log N / \log L} =2 C_5^{- d_f} N^{-M+1} \left|x-y\right|^{d_f} .
\end{equation*}

\smallskip

  \textsc{Case 3.} If $\Delta_{M_B-1}(x) \cap \Delta_{M_B-1}(y) \neq \emptyset$ and $\Delta_{M}(x) \cap \Delta_{M}(y) \neq \emptyset$, then, since it is assumed $d_M(x,y)>1,$ we have
\begin{equation*}
d_M(x,y) = 2.
\end{equation*}

 Therefore the inequality holds with $C_6 (M) = L^{-M}/ \textrm{diam}\left(\Delta_0\right) $, $C_7(M) = 2N^{-M+1}C_5^{- d_f}$. The proof for $x$ or $y$ in $V^{\left\langle \infty \right\rangle}_{M}$ comes analogously.
\end{proof}

{We are now in a position to show that there exists $n \in \mathbb{N}$ such that for every $M \in \mathbb{Z}$, if $x,y\in \mathcal{K}^{\left\langle \infty \right\rangle}$ satisfy $|x-y|\leq L^M$, then $d_M(x,y)\leq n$ (recall that under this assumption the two-sided estimates
\ref{eq:kum} has been proven in \cite{bib:Kum}). Indeed, from Lemma \ref{lem:metrics} we get
\begin{equation} \label{eq:ass_Kum}
d_M(x,y) \leq 2 \vee \left(2N^{-M+1}C_5^{-d_f} \left|x-y\right|^{d_f}\right)
\leq 2 \vee \left(2N C_5^{-d_f} N^{-M} L^{Md_f}\right) = 2 \vee \left(2N C_5^{-d_f}\right).
\end{equation}
Hence we can simply take $n=\max\left\{2,2N C_5^{-d_f}\right\}$ (uniformly in $M$).
}


We also needed the following lemma, giving the upper estimate of the cardinality of $\mathcal{L}_{n,x},$ introduced in \ref{eq:xComplexdist}. Informally speaking, this is the number of $M-$cells lying at $d_M-$distance $n$ from the point $x.$

\begin{lemma}
\label{lem:complab}
There exists a universal constant $C_{8}$ such that for any $M \in \mathbb{Z}$ and $x \in \mathcal{K}^{\left\langle \infty \right\rangle}$
\begin{equation*}
\# \mathcal{L}_{M,n,x} \leq C_{8} n^{d_f}.
\end{equation*}
\end{lemma}

\begin{proof}
Let  $M \in \mathbb{Z}$ be fixed. The lemma follows from the comparison of the Euclidean distance on the plane and the $M$-graph distance. Notice that if we pick one vertex from each $M$-complex in the way that all these vertices have the same alignment with respect to the $M$-complex (e.g. we can choose the leftmost of the lowest vertices of each $M$-complex), then we get the collection of points, exactly one in each $\Delta_M \in \mathcal{T}_M$, mutually at Euclidean distance greater than or equal to $L^M c_1$, where $c_1 = \inf \left\{ \left\| \nu_i-\nu_j \right\|: 1\leq i,j \leq N, i\neq j\right\}.$

Also notice that if $y \in \Delta_M \in \mathcal{L}_{M,n,x}$, then
$$
d_M(x,y) \leq n.
$$

Using the inequality
$$
\left|x-y\right| \leq \frac{1}{C_6(M)} d_M(x,y)
$$
we get that all $M$-complexes from $\mathcal{L}_{M,n,x}$ are included in the ball
$$
\left\{y\in \mathbb{R}^2: \left|x-y\right| \leq \frac{ n}{C_6(M)} \right\}.
$$

Let us now estimate how many points which are mutually at distance greater or equal to $L^M c_1$ can be packed into such ball. It is limited by the  ratio of Hausdorff-$d_f$measures of a ball with radius $\frac{n}{C_6(M)}+\frac{L^M c_1}{2} =n L^M \diam\left(\Delta_0\right)+\frac{L^M c_1}{2}$ and a ball with radius $\frac{L^M c_1}{2}$ (the radius of the bigger ball is increased as some points we picked might lie close to the boundary of the ball). Finally we get
$$
n_{max} \leq \frac{c_2 \left( n L^M \diam\left(\Delta_0\right)+\frac{L^M c_1}{2}\right)^{d_f}}{c_{3} \left(\frac{L^M c_1}{2} \right)^{d_f} } = c_4 \left(\frac{2n \diam\left(\Delta_0\right)}{ c_1} +1\right)^{d_f} \leq C_8 n^{d_f}
$$
for a sufficiently large constant $C_{8}$, independent of $M \in \mathbb{Z}$.
\end{proof}

\bigskip

\end{document}